\documentclass[12pt]{amsart}
\usepackage{amssymb}
\usepackage{amsthm,amsfonts}
\usepackage{amsfonts}
\usepackage{amsmath}
\usepackage[initials]{amsrefs}
\usepackage{enumerate}
\usepackage{color}
\usepackage[colorlinks,linkcolor=blue]{hyperref}
\usepackage[nameinlink]{cleveref}

\usepackage{geometry}
\newtheorem{Theorem}{Theorem}[section]

\newtheorem{Lemma}[Theorem]{Lemma}
\newtheorem{Corollary}[Theorem]{Corollary}

\newtheorem{Definition}[Theorem]{Definition}

\newtheorem{corollary}[Theorem]{Corollary}

\newtheorem{definition}[Theorem]{Definition}
\theoremstyle{definition}

\newcommand{\R}{\mathbb{R}}

\newcommand{\N}{\mathbb{N}}

\newcommand{\C}{\mathbb{C}}

\newcommand{\var}{\varepsilon}

\DeclareMathOperator{\supp}{supp}

\begin{document}
\title{New results in analysis of Orlicz-Lorentz spaces}
\author[Bernal]{Luis Bernal-Gonz\'alez}
\address[L. Bernal-Gonz\'alez]{\mbox{}\newline \indent Departamento de An\'{a}lisis Matem\'{a}tico \newline \indent Facultad de Matem\'{a}ticas \newline \indent
	Instituto de Matem\'aticas Antonio de Castro Brzezicki \newline \indent
	Universidad de Sevilla \newline \indent
	Avenida Reina Mercedes, Sevilla, 41012, Spain.}
\email{lbernal@us.es}
\author[Rodriguez]{Daniel L. Rodr\'iguez-Vidanes}
\address[D.L.~Rodr\'iguez-Vidanes]{\mbox{}\newline \indent Instituto de Matem\'atica Interdisciplinar (IMI)\newline \indent Departamento de An\'{a}lisis Matem\'{a}tico y Matem\'atica Aplicada \newline \indent
	Facultad de Ciencias Matem\'{a}ticas\newline \indent
	Plaza de Ciencias 3 \newline \indent
	Universidad Complutense de Madrid \newline \indent
	Madrid, 28040, Spain.}
\email{dl.rodriguez.vidanes@ucm.es}
\author[Seoane]{Juan B.~Seoane-Sep\'{u}lveda}
\address[J.B.~Seoane-Sep\'{u}lveda]{\mbox{}\newline \indent Instituto de Matem\'atica Interdisciplinar (IMI)\newline \indent Departamento de An\'{a}lisis Matem\'{a}tico y Matem\'atica Aplicada \newline \indent
	Facultad de Ciencias Matem\'{a}ticas\newline \indent
	Plaza de Ciencias 3 \newline \indent
	Universidad Complutense de Madrid \newline \indent
	Madrid, 28040, Spain.}
\email{jseoane@mat.ucm.es}
\author[Tag]{Hyung-Joon Tag}
\address[Hyung-Joon Tag]{\mbox{}\newline \indent Department of Mathematics Education
	\newline \indent Dongguk University \newline \indent
	1 gil Pildong-ro, Jung-gu \newline \indent
	Seoul, 04620, Republic of Korea.}
\email{hjtag4@gmail.com}
\thanks{
The research of the first author has been supported by the Plan Andaluz de
Investigaci\'on de la Junta de Andaluc\'{\i}a FQM-127 and by FEDER grant US-1380969.
The second author has been supported by Grant PGC2018-097286-B-I00 and by the Spanish Ministry of Science, Innovation and Universities and the European Social Fund through a “Contrato Predoctoral para la Formación de Doctores, 2019” (PRE2019-089135).
}
\maketitle
\begin{abstract}
In this article, we investigate the existence of closed vector subspaces (i.e.~spaceability) in various nonlinear subsets of Orlicz-Lorentz spaces \,$\Lambda_{\varphi,w}$\, equipped with the Luxemburg norm. If a family of Orlicz functions \,$(\varphi_n)_{n=1}^{\infty}$\, sa\-tis\-fies certain order relations with respect to a given Orlicz function \,$\varphi$, the subset of the order-continuous subspace $(\Lambda_{\varphi,w})_a$ whose elements do not belong to \,$\bigcup_{n=1}^{\infty}\Lambda_{\varphi_n,w}$\, is spaceable, and even maximal-spaceable when $\varphi$ satisfies the $\Delta_2$-condition. We also show that this subset is either residual or empty. In addition, sufficient conditions for this subset not being $(\alpha, \beta)$-spaceable are provided. A similar analysis is also performed on the subset $\Lambda_{\varphi,w} \setminus (\Lambda_{\varphi,w})_a$ when $\varphi$ does not satisfy the $\Delta_2$-condition.

The comparison between different Orlicz-Lorentz spaces is characterized via the generating pairs \,$(\varphi,w)$. For a fixed Orlicz function that satisfies the $\Delta_2^{\infty}$-condition, we provide a characterization of disjointly strictly singular inclusion operators between Orlicz-Lorentz spaces with different weights. As a consequence, there are certain subsets of Orlicz-Lorentz spaces on $[0,1]$ for which the lineability problem is not valid. Moreover, various types of $(\alpha,\beta)$-lineability and pointwise lineability properties on other nonlinear subsets of Orlicz-Lorentz spaces are examined. These results extend a number of previously known results in Orlicz and Lorentz spaces.

	
\end{abstract}

\section{Introduction}

The study of lineability, a term introduced by V.\,I\,.~Gurariy \cite{AGS,Tesis}, concerns the existence of vector spaces as well as other algebraic structures in sets that may not form a vector space.
For a deeper understanding on this topic, we refer the interested reader to \cite{ABPS,TAMS2014,BAMS2014,Studia2017,TAMS2020,BAMS2019}.
Let us introduce the basic notion (see \cite{ABPS}).


\begin{Definition}
Let \,$X$ \,be a vector space over a field, \,$A \subset X$ and \,$\alpha$ a cardinal number.
The set \,$A$ \,is called \,\emph{$\alpha$-lineable} provided that there is
a vector subspace \,$M \subset X$\, of dimension $\alpha$ \,such that \,$M \subset A \cup \{0\}$.
Moreover, if $\alpha=\dim (X)$, where \,$\dim(X)$ denotes the algebraic dimension of $X$, then we say that $A$ is \emph{maximal-lineable}.
\end{Definition}

In this article, we focus on Orlicz-Lorentz spaces.
This class is a natural generalization of both Orlicz and Lorentz spaces and includes the Lebesgue spaces \,$L_{p},\,  1 \leq p < \infty$.
Here, the support for these spaces will be an interval \,$[0,\gamma ) \subset \R$, where \,$0 < \gamma \le \infty$.
As a matter of fact, large {\it closed} vector spaces can be extracted from certain nonlinear subsets of appropriate Orlicz-Lorentz spaces.
We formally present the corresponding concept (see \cite{ABPS}):


\begin{Definition}
Let \,$X$ \,be a topological vector space over a field, \,$A \subset X$ and let \,$\alpha$\, be a cardinal number.
The set \,$A$ \,is called:
	\begin{enumerate}[\rm(i)]
		\item \emph{spaceable} provided that there is an infinite-dimensional closed vector subspace \,$M$ \,of \,$X$ \,such that \,$M \subset A \cup \{0\}$.
		\item \emph{$\alpha$-spaceable} provided that there is an $\alpha$-dimensional closed vector subspace \,$M$ \,of \,$X$ \,such that \,$M \subset A \cup \{0\}$.
		\item \emph{maximal-spaceable} provided that \,$M$ is $\dim (X)$-spaceable.
		\item \emph{$\alpha$-dense-lineable} provided that there is an $\alpha$-dimensional dense vector subspace \,$M$ \,of \,$X$ \,such that \,$M \subset A \cup \{0\}$.
		\item \emph{maximal dense-lineable} provided that \,$M$ is $\dim (X)$-dense-lineable.
	\end{enumerate}
\end{Definition}

In 2012, it was proved by Botelho {\it et al.}~\cite{BFPS} that, under the natural topology on \,$L_p[0,1]$, the set \,$L_p[0,1] \setminus \bigcup_{q > p} L_q[0,1]$ \,is spaceable for every \,$p > 0$. Also, under suitable conditions on the corresponding measure space \,$(\Omega , \Sigma , \mu )$, the first author and Ord\'o\~{n}ez \cite{BO} provided characterizations for the spaceability of the families \,$L_p(\Omega ) \setminus \bigcup_{q \in [1,p)} L_q(\Omega )$ $(p > 1)$, $L_p(\Omega ) \setminus \bigcup_{q \in (p,\infty )} L_q(\Omega )$ $(p \ge 1)$ \,and
 \,$L_p(\Omega ) \setminus \bigcup_{q \in [1,\infty) \setminus \{p\}} L_q(\Omega )$ $(p > 1)$ \,(see also \cite{BO2}).
 Moreover, the $(\alpha,\beta)$-spaceability of the latter families has been thoroughly studied in \cite{ABRR2}.

\vskip 3pt

The assertions in the preceding paragraph have been extended in several directions. On the one hand, in 2014, Ruiz and S\'anchez \cite[Proposition 5.2]{RS} obtained for the Orlicz spaces \,$L_\varphi$ \,on \,$[0,1]$ \,(where $\varphi$ is an Orlicz function, see Section 2) that \,$L_\varphi \setminus \bigcup_{n=1}^\infty L_{\varphi_n}$ \,is spaceable under the natural topology of \,$L^\varphi$, where \,$(\varphi_n)_{n=1}^\infty$ \,is a sequence of Orlicz functions such that each inclusion \,$L_{\varphi_n} \hookrightarrow L_\varphi$ \,is continuous and \,$L_{\varphi_n} \not = L^\varphi$ \,for every \,$n \in \N$.

\vskip 3pt

On the other hand, Akbarbaglu and Maghsoudi \cite[Theorems~3.1 and~3.2]{AM} were able to provide, for rather general measure spaces \,$(\Omega , \Sigma , \mu )$ \,enjoying appropriate natural properties, sufficient conditions for the spaceability of \,$M^\varphi (\Omega ) \setminus \bigcup_{n \ge 1} L_{\varphi_n}(\Omega )$, where
\,$\varphi , \varphi_n$ \,are Orlicz functions \,and \,$M^\varphi (\Omega )$ \,stands for the order-continuous subspace of \,$L_\varphi(\Omega )$ (see Section 2 for their definitions).

\vskip 3pt

Our aim in this paper is to extend the spaceability results of the last two paragraphs to a class of spaces containing both Lorentz and Orlicz spaces, namely, the family of Orlicz-Lorentz spaces on the real interval $[0, \gamma)$, where $\gamma \leq \infty$. Furthermore, the recently introduced notions of $(\alpha,\beta)$-lineable, $(\alpha,\beta)$-spaceable and pointwise $\alpha$-lineable will also be studied in the context of Orlicz-Lorentz spaces.

This article is organized as follows. Section~2 is devoted to settle the main concepts and tools in order to establish the desired spaceability assertions. In Section~3, inclusion relations between different Orlicz-Lorentz spaces are characterized via the generating pairs \,$(\varphi,w)$ (Theorems~\ref{Thm Orl-Lor inclusion-2} and~\ref{Thm Orl-Lor inclusion}). Moreover, for a fixed Orlicz function $\varphi$ and two different weights, we provide a characterization of inclusion operators between these spaces being disjointly strictly singular (Theorem \ref{Thm OL DSS characterization}). This consequently shows that certain subsets of Orlicz-Lorentz spaces are empty (Theorem~\ref{Thm empty difference}), and so the lineability problem for those subsets is not valid. In Section~4, we examine various notions of spaceability and lineability on certain nonlinear subsets of Orlicz-Lorentz spaces. If a family of Orlicz functions $(\varphi_n)_{n=1}^{\infty}$ satisfies certain order relations with respect to $\varphi$, then the subset of order-continuous elements in $\Lambda_{\varphi,w}$ that are not contained in $\Lambda_{\varphi_n,w}$ for every $n \in \mathbb{N}$ is spaceable, and even maximal-spaceable if $\varphi$ additionally satisfies the appropriate $\Delta_2$-condition (Theorems~\ref{th:OLspaceable1}, \ref{th:OLspaceable2} and \ref{co:OLspaceablemix}). These subsets are always either residual or empty (Theorem~\ref{th:OLresidual}). We will also provide sufficient conditions for the same subsets being not $(\alpha, \beta)$-spaceable (Corollaries~\ref{co:OLspaceable1}, \ref{co:OLspaceable2} and \ref{co:OLspaceable3}).
If $\varphi$ does not satisfy the appropriate $\Delta_2$-condition, then the complement of the order-continuous subspace in $\Lambda_{\varphi,w}$ is spaceable and contains an isometric copy of $\ell_{\infty}$ (Theorem \ref{Thm isomorphic copy}). Finally, under certain assumptions, we examine various lineability and spaceability properties on the set $\Lambda_{\varphi,w} \setminus \bigcup_{n=1}^{\infty}\Lambda_{\varphi,w_n}$ when $\varphi$ satisfies the appropriate $\Delta_2$-condition (Theorem~\ref{thm:OLspaceabilityspan1}) and when the family $(\varphi_n)_{n=1}^{\infty}$ satisfies certain order relations with respect to $\varphi$ (Theorem \ref{thm:OLspaceabilityspan2} and Corollary \ref{cor:OLspaceabilityspan2}).



\section{Preliminaries}

\subsection{Orlicz-Lorentz spaces}
Let \,$\Omega = (\Omega, \Sigma, \mu)$ \,be a measure space. The set of all real $\mu$-measurable functions on $\Omega$ modulo a.e.-convergence is denoted by \,$L_0$
\,(so two measurable functions are considered to be equal if they coincide outside of a $\mu$-null set). Recall that a function $\Omega \to \R$ is called $\mu$-measurable if it is
the $\mu$-a.e.~pointwise limit of a sequence of simple functions supported on $\mu$-finite sets; 
see, e.g.,~\cite{HNVW} for definitions and basic properties.

\vskip 3pt

Here we only consider $\Omega = I= [0, \gamma)$, $\gamma \leq \infty$, and the $\sigma$-algebra $\Sigma$ of Lebesgue measurable subsets in $I$ with the Lebesgue measure $\mu = m$
(so a function $f$ being $\mu$-measurable is equi\-va\-lent to being Lebesgue measurable, in other words, $f^{-1}(J)$ is a Lebesgue measurable set for each interval $J$). A Banach space $(X, \|\cdot\|) \subset L_0$ is said to be a {\it Banach function space} (more precisely, a {\it solid Banach space}\hskip 1pt) if $|f| \leq |g|$ implies $\|f\|_X \leq \|g\|_X$.
A Banach function space $X$ is said to have the {\it Fatou property} if every sequence $(f_n)_{n=1}^\infty \subset X$ such that $|f_n| \uparrow |f|$ with $\sup_n\|f_n\|_X < \infty$, we have $f \in X$ and $\|f_n\|_X \rightarrow \|f\|_X$.
A function $f \in X$ is said to be {\it order-continuous} if $\|f_n\|_X\downarrow 0$ for every sequence $(f_n)_{n=1}^{\infty}$ of measurable functions such that $|f_n| \leq |f|$ for every $n \in \mathbb{N}$ and
$|f_n| \downarrow 0$ $\mu$-a.e. The space of order-continuous elements in $X$ is denoted by $X_a$.
We denote the closure of the set of all simple functions with supports having finite measure by $X_b$. In general, it is well-known that \,$X_a \subset X_b$ \cite[Theorem I.3.11]{BS}. A closed linear subspace $Y$ of a Banach function space $X$ is an {\it order-ideal} \,if $f \in Y$ and $|g| \leq |f|$ $\mu$-a.e. implies $g \in Y$. The order-continuous subspace of a Banach function space is an order-ideal \cite[Theorem I.3.8]{BS}.

\vskip 3pt

For $\lambda > 0$ and $f \in L_0$, the {\it distribution function} \,$d_f(\lambda)$ \,is defined by $\,d_f(\lambda) = \mu(\{t \in \Omega : |f(t)| > \lambda\})$. We say that $f, g \in L_0$ are {\it equimeasurable} if $d_f(\lambda) = d_g(\lambda)$ for all $\lambda > 0$. The {\it decreasing rearrangement} \,$f^*$ \,of \,$f \in L_0$, is the generalized inverse of \,$d_f(\lambda)$ \,defined by \,$f^*(t) = \inf\{\lambda > 0: d_f(\lambda) \leq t\}$ for $t \geq 0$. A function \,$f \in L_0$ \,and its decreasing rearrangement \,$f^*$ \,are known to be equimeasurable. A Banach function space is \,{\it rearrangement invariant} \,whenever \,$\|f\|_X = \|g\|_X$ \,for equimeasurable functions \,$f, g \in X$. Orlicz-Lorentz function spaces are rearrangement invariant Banach function spaces with the Fatou property. For more details on Banach function spaces, we refer to \cite{BS, LT2}.

\vskip 3pt

An {\it Orlicz function} (also called {\it Young function}) $\varphi: [0,\infty ) \rightarrow [0,\infty )$ is a convex function such that \,$\varphi(0) = 0 = \lim_{u \to 0} \varphi (u)$ \,and $\varphi(u) > 0$ for every $u > 0$. Note that an Orlicz function is strictly increasing, continuous, and unbounded.
In the theory of Orlicz spaces, the following growth condition on Orlicz functions and the orders between Orlicz functions play important roles.
	
\begin{Definition}
	${}$
		\begin{enumerate}[\rm(i)]
			\item An Orlicz function \,$\varphi$ \,{\rm satisfies the $\Delta_2^{\infty}$-condition {\rm (}resp.~$\Delta_2^0$-con\-di\-tion{\rm )}} if there exists \,$K > 2$ and $u_0 \geq 0$
			\,such that \,$\varphi(2u) \leq K\varphi(u)$ for every $u \geq u_0$ {\rm(resp.}~$u \leq u_0${\rm)}. If $\varphi$ satisfies both $\Delta_2^{\infty}$-condition and $\Delta_2^{0}$-condition, then we say that $\varphi$ has the {\rm $\Delta_2$-condition}.
			\item For two Orlicz functions \,$\varphi$ \,and \,$\psi$, the function \,$\psi$ \,{\rm satisfies the $\Delta_\varphi(\infty)$-condition} {\rm (resp.~the $\Delta_\varphi(0)$-condition)} if for every \,$b > 0$ \,there exists \,$u_0$ \,such that \,$\varphi(u) \leq \psi(b u)$ \,for every \,$u \geq u_0$
			{\rm (}resp.~$0 \leq u \leq u_0${\rm ).}
			\item For two Orlicz functions \,$\varphi$ \,and \,$\psi$, we denote the order \,$\varphi \prec \psi$ \,{\rm (}resp.~$\varphi \prec_{\infty} \psi${\rm )}
\,if there exists \,$b > 0$
\,{\rm (}resp.~$b > 0$ \,and \,$u_0 \geq 0${\rm )} \,such that \,$\varphi(u) \leq \psi(bu)$ \,for every \,$u \geq 0$ {\rm (}resp.~$u \geq u_0${\rm ).}
		\end{enumerate}
\end{Definition}

It is easy to see from the definitions that if \,$\psi$ \,satisfies the $\Delta_{\varphi}(\infty)$-condition, then $\varphi \prec_{\infty} \psi$.
Also, we mention that an Orlicz function $\varphi$ satisfying the appropriate $\Delta_2$-condition characterizes the separability of an Orlicz-Lorentz space \cite{K}.
	
In this article, the following equivalent statement of the $\Delta_2$-conditions will be used.
	

\begin{Lemma} \cite[Theorem 1.13] {C}  \label{le:C}
		An Orlicz function $\varphi$ satisfies the $\Delta_2$-
		{\rm (}resp.~$\Delta_2^\infty$-{\rm )} condition if and only if there exist \,$l>1$ \,and \,$K>1$
\,{\rm (}resp.~$l>1$, $K>1$,  $u_0\ge 0${\rm )} \,such that \,$\varphi(lu) \leq K \varphi(u)$ \,for all \,$u\ge 0$
\,{\rm (}resp.~$u\ge u_0${\rm )}.
	\end{Lemma}

	

A {\it weight function} \,$w: I \rightarrow \R$ \,is a decreasing, positive, locally integrable function. We denote \,$W(t) := \int_0^t w(s) \,ds$. Of course, $W(0) = 0$.
We assume that \,$W(\infty) = \int_0^{\infty} w(s) \,ds = \infty$ \,in the case \,$I = [0,\infty )$. The {\it modular} corresponding to $(\varphi ,w)$ is the function $\rho_{\varphi,w}: L_0 \rightarrow [0, \infty]$ defined by
\[
\rho_{\varphi,w}(f) = \int_I \varphi(f^*) w = \int_I \varphi(f^*(t))w(t) \,dt.
\]
From the fact that $(f^*)^* = f^*$, we see that $\rho_{\varphi,w}(f) = \rho_{\varphi,w}(f^*)$ for every $f \in L_0$. The modular $\rho_{\varphi,w}$ is {\it orthogonally subadditive}, that is, $\rho_{\varphi,w}(f + g) \leq \rho_{\varphi,w}(f) + \rho_{\varphi,w}(g)$ for every pair $f, g \in L_0$ with $f \wedge g = 0$ (disjointly supported). Also, if $|f| \leq |g|$, then $\rho_{\varphi,w}(f) \leq \rho_{\varphi,w}(g)$; in other words, the modular $\rho_{\varphi,w}$ is {\it monotone}.

\vskip 3pt

For every pair of functions \,$(\varphi , w)$ \,as above, the corresponding {\it Orlicz-Lorentz space} $\Lambda_{\varphi,w}$ is defined by
\[
\Lambda_{\varphi,w} = \{f \in L_0: \rho_{\varphi,w}(kf) < \infty\,\,\, \text{for some}\,\,\, k > 0\}.
\]
We consider the Luxemburg norm $\|\cdot\|_{\varphi, w}$ on $\Lambda_{\varphi,w}$, that is given by
\[
\|f\|_{\varphi,w} = \inf\left\{\var > 0: \rho_{\varphi,w}\left(\frac{f}{\var}\right)\leq 1\right\}.
\]
From now on, we assume that $\Lambda_{\varphi,w}$ is equipped with the Luxemburg norm. Then \,$(\Lambda_{\varphi,w}, \|\cdot\|_{\varphi,w})$ \,is a Banach space
(in fact, it is a Banach lattice; see, e.g., \cite{Mast}). It is well-known (see \cite{K}) that the order-continuous subspace $(\Lambda_{\varphi,w})_a$ of an
Orlicz-Lorentz space can be expressed by
\[
(\Lambda_{\varphi,w})_a = (\Lambda_{\varphi,w})_b = \{f \in L_0: \rho_{\varphi,w}(kf) < \infty\,\,\, \text{for every}\,\,\, k > 0\}.
\]

We mention that the Orlicz-Lorentz space becomes the Lorentz space $\Lambda_{p,w}(I)$ if $\varphi(u) = u^p$, the Orlicz space $L_{\varphi}(I)$ if $w \equiv 1$, and $L_p$ if we let $\varphi(u) = u^p$ and $w \equiv 1$.

\subsection{``Extendable'' and pointwise lineability}

We have seen that the $\alpha$-lineability ($\alpha$-spaceability) on a nonlinear subset of a topological vector space $X$ concerns the existence of (closed) $\alpha$-dimensional vector subspaces. Now, for an $\alpha$-lineable ($\alpha$-spaceable) set $A \subset X$, let us consider all $\alpha$-dimensional vector subspaces of $X$ in $A$. Then, several natural questions arise. First, we can ask for the existence of $\beta$-dimensional (closed) vector subspaces of $X$ in $A$, $\beta \geq \alpha$, that contain those $\alpha$-dimensional vector subspaces. Also, it is natural to ask how large the dimension $\beta$ can be if such subspaces exist. On the other hand, for each point in $A$, we can also question the existence of an $\alpha$-dimensional vector subspace of $X$ in $A$ that contains the point. These questions led to the following properties (see \cite{CGP,FPRR,FPT, PR}).

\begin{definition}
	Let $X$ be a topological vector space defined over a field, $A \subset X$ and $\alpha \leq \beta$ cardinal numbers.
	We say that $A$ is:
		\begin{enumerate}[\rm(i)]
			\item \emph{$(\alpha,\beta)$-lineable} if $A$ is $\alpha$-lineable and for every $\alpha$-dimensional vector subspace $M_\alpha$ of $X$ with $M_\alpha \subset A\cup \{0\}$ there is a $\beta$-dimensional vector subspace $M_\beta$ of $X$ such that $M_\alpha \subset M_\beta \subset A\cup \{0\}$.
			\item \emph{$(\alpha,\beta)$-spaceable} if $A$ is $\alpha$-lineable and for every $\alpha$-dimensional vector subspace $M_\alpha$ of $X$ with $M_\alpha \subset A\cup \{0\}$ there is a closed $\beta$-dimensional vector subspace $M_\beta$ of $X$ such that $M_\alpha \subset M_\beta \subset A\cup \{0\}$.
			\item \emph{$(\alpha,\beta)$-dense-lineable} if $A$ is $\alpha$-lineable and for every $\alpha$-dimensional vector subspace $M_\alpha$ of $X$ with $M_\alpha \subset A\cup \{0\}$ there is a dense $\beta$-dimensional vector subspace $M_\beta$ of $X$ such that $M_\alpha \subset M_\beta \subset A\cup \{0\}$.
			\item \emph{pointwise $\alpha$-lineable} if for every $x \in A$ there is an $\alpha$-dimensional vector subspace $M_{\alpha,x}$ of $X$ such that $x \in M_{\alpha,x} \subset A\cup \{0\}$.
			\item \emph{pointwise maximal-lineable} if $A$ is \emph{pointwise $\dim(X)$-lineable}.
			\item \emph{pointwise $\alpha$-dense-lineable} if for every $x \in A$ there is an $\alpha$-dimensional dense vector subspace $M_{\alpha,x}$ of $X$ such that $x \in M_{\alpha,x} \subset A\cup \{0\}$.
			\item \emph{pointwise maximal dense-lineable} if $A$ is \emph{pointwise $\dim(X)$-dense-lineable}.
		\end{enumerate}
\end{definition}

On the one hand, the $(\alpha,\beta)$-lineability (resp. $(\alpha,\beta)$-spaceability, $(\alpha,\beta)$-dense-lineablity) can be seen as an extendability property of the $\alpha$-lineability (resp. $\alpha$-spaceability, and $\alpha$-dense-lineability).
On the other hand, the pointwise $\alpha$-lineability (resp. pointwise maximal-lineability, pointwise $\alpha$-dense-lineability, pointwise maximal dense-lineability) can be described as the local version of the $\alpha$-lineability (resp. maximal-lineability, $\alpha$-dense-lineability, maximal dense-lineability).

The $(\alpha,\alpha)$-lineability is equivalent to $\alpha$-lineability. From the fact that $(0,\beta)$-lineability (spaceability) is equivalent to the $\beta$-lineability (spaceability), the $(\alpha, \beta)$-lineability (spaceability) implies the $\beta$-lineability (spaceability).
Similarly, the $(\alpha, \alpha)$-spaceability implies the $\alpha$-spaceability, but the converse is not true in general.
Indeed, as stated earlier $L_p[0,1] \setminus \bigcup_{q >p} L_q [0,1]$ is $\mathfrak c$-spaceable for any $p>0$, but as an immediate consequence of \cite[Corollary~2.4]{FPRR}, we obtain that $L_p[0,1] \setminus \bigcup_{q>p} L_q [0,1]$ is not $(\mathfrak c,\mathfrak c)$-spaceable. In fact, Araújo {\it et al.} proved in \cite[Theorem~3]{ABRR} that $L_p[0,1] \setminus \bigcup_{q>p} L_q [0,1]$ is $(\alpha,\mathfrak c)$-spaceable if and only if $\alpha < \aleph_0$ (see also \cite[Corollary~10]{ABRR2}). If $\alpha_1<\alpha_2$, then the $\alpha_2$-lineability implies the $\alpha_1$-lineability. However, there is no analogous statement for the $(\alpha,\beta)$-lineability property. Precisely, if $\alpha_1<\alpha_2\leq \beta$, then $(\alpha_2,\beta)$-lineability does not imply $(\alpha_1,\beta)$-lineability and vice versa (see \cite{FPT}).
Finally, pointwise $\alpha$-lineability implies $(1,\alpha)$-lineability, but the reverse implication is not necessarily true (see \cite{PR}).



\subsection{Tools for spaceability in a vector space} For a Banach space \,$X$ \,over \,$\mathbb F$ ($=\R$ or $\C$), a sequence $(x_n)_{n=1}^{\infty}$ is said to be {\it basic} if it is a Schauder basis for $\overline{\rm span}\{x_n : n\in \mathbb N \}$. It is well-known that a sequence $(x_n)_{n=1}^\infty$ is basic if, and only if, there is a constant $C> 0$ such that for every $r,s \in \mathbb{N}$ where $r \geq s$ and every finite sequence of scalars $a_1, \dots, a_r$, we have  $\left\|\sum_{n=1}^{s}a_n x_n\right\|_X\leq C\left\|\sum_{n=1}^{r}a_n x_n\right\|_X$. A subset $S \in X$ is a {\it cone} if \,$c\,S \subset S$ \,for all \,$c \in \mathbb{F}$.

\vskip 3pt

A topological vector space $X$ of $\mathbb{F}$-valued functions on $\Omega$ is a {\it PCS-space} if for a given sequence $(f_n)_{n=1}^\infty \subset X$, the existence of $f \in X$ with $f_n \rightarrow f$ in $X$ implies the existence of a subsequence $(n_k)_{k=1}^\infty \subset \mathbb{N}$ depending on $(f_n)_{n=1}^\infty$ such that for every $t \in \Omega$, $f_{n_k}(t) \rightarrow f(t)$ as $k \rightarrow \infty$. We have the following relationship between PCS-spaces and their spaceability.

\begin{Theorem}\label{th:pcs}\cite[Theorem 2.2]{BO}
	Let $\Omega$ be a nonempty set. Assume that $(X, \|\cdot\|_X)$ is a Banach space of $\mathbb{F}$-valued functions on $\Omega$ and that $B$ is a nonempty subset of $X$ satisfying the following properties:
	\begin{enumerate}[\rm(i)]
		\item $X$ is a PCS-space.
		\item There is a constant $C \in (0, \infty)$ such that $\|f + g\| \geq C\|f\|$ for all $f,g \in X$ with $\supp f \cap \supp g = \varnothing$.
		\item $B$ is a cone.
		\item If $f, g \in X$ such that $f + g \in B$ and $\supp f \cap \supp g = \varnothing$, then $f,g \in B$.
		\item There exists a sequence of functions $(f_n)_{n=1}^{\infty}$ with pairwise disjoint supports such that, for every $n \in \mathbb{N}$, $f_n \in X \setminus B$.
	\end{enumerate}
	Then $X \setminus B$ is $\mathfrak c$-spaceable.
\end{Theorem}

\vskip -5pt

It is important to mention that the original conclusion of Theorem~\ref{th:pcs} states that \,$X\setminus B$\, is spaceable.
However, a careful inspection on the proof of Theorem \ref{th:pcs} reveals that the sequence of functions $(f_n)_{n=1}^{\infty}$ from (v) forms, in fact, a basic sequence in $X$.
Hence, we see that $\overline{\rm span}\{f_n : n\in \mathbb N \}$ is a closed vector subspace contained in $X\setminus B$; which implies precisely the $\mathfrak c$-spaceability but not necessarily $\kappa$-spaceability property for some cardinal number $\kappa >\mathfrak c$.
Also, the conditions (i) and (ii) always hold for Banach function spaces. Indeed, from the fact that functions in a Banach function space $X$ are equal with respect to the $\mu$-a.e. equivalence relation, the condition (i) holds. The condition (ii) directly comes from the definition of a Banach function space.


\vskip 3pt

We can also observe the spaceability of rearrangement invariant Banach function spaces $X$ over $[0,1]$ through $T$-subsets \cite{RS}. For \,$a \in [0,1)$, $b \in (0, 1- a]$ \,and a measurable function \,$f$ \,on \,$[0,1]$, consider the linear operator \,$f \mapsto T_{a,b} f$ \,given by \,$(T_{a,b}f) (t) = f \left(\frac{t-a}{b}\right)\chi_{(a, a + b]}(t)$. From the fact that this operator is bounded from $L_1$ to itself and from $L_{\infty}$ to itself, it is also bounded from any rearrangement invariant Banach function space $X$ to itself by Calder\'on-Mitjagin interpolation theorem \cite[Theorem 2.a.10]{LT2}. Then, for a rearrangement invariant Banach function space $X$ on $[0,1]$, a subset $E \subset X$ is said to be a {\it $T$-subset} if it satisfies:
	\begin{enumerate}[\rm(i)]
		\item If $g \in E$ and $|f(t)| \leq |g(t)|$ $\mu$-a.e. on $[0,1]$, then $f \in E$.
		\item If there exists $a \in [0, 1)$ and $b \in (0, 1-a]$ such that $T_{a,b}(f) \in E$, then $f \in E$.
	\end{enumerate}
	
We have the following characterization for the complement of a $T$-subset to be spaceable.
	
\begin{Theorem}\cite[Theorem 3.1]{RS}
		Let \,$X$ be a rearrangement invariant Banach function space on $[0,1]$ and $E \subset X$ be a $T$-subset. The following statements are equivalent:
		\begin{enumerate}[\rm(i)]
			\item $X \setminus E$ is spaceable.
			\item $X \setminus E$ is nonempty.			
		\end{enumerate}
\end{Theorem}

\noindent If the space $X$ is separable, we have something more:

\begin{Corollary}\cite[Corollary 3.2]{RS}\label{lem:inclusion}
		Let \,$X$ be a separable rearrangement invariant Banach function space and $(X_i)_{i \in I}$ be a collection of rearrangement invariant Banach function spaces contained in $X$. The following statements are equivalent:
		\begin{enumerate}[\rm(i)]
			\item $X \setminus \bigcup_{i \in I} X_i$ is spaceable.
			\item $X \setminus \bigcup_{i \in I} X_i$ is nonempty.
			\item $\bigcup_{i \in I} X_i$ is not closed in $X$.
		\end{enumerate}
	\end{Corollary}


\section{On the inclusion operator between  Orlicz-Lorentz spaces}
In this section, we study the inclusion operators between Orlicz-Lorentz spaces. Since Orlicz-Lorentz spaces $\Lambda_{\varphi,w}$ are defined by an Orlicz function $\varphi$ and a weight function $w$, we will fix one of the parameters and provide a characterization of the inclusion relation between these spaces respectively. These characterizations will play an important role to show the emptiness of certain subsets of Orlicz-Lorentz spaces in this section and to study extendable and pointwise lineability on certain subsets of Orlicz-Lorentz spaces in Section 4.

\vskip 3pt

For the first point, since the union of rearrangement invariant Banach function spaces contained in another rearrangement invariant Banach function space is a $T$-subset, we will consider disjointly strictly singular inclusion operators from an Orlicz-Lorentz space to another one in view of Corollary \ref{lem:inclusion}. For a Banach function lattice $X$ and a Banach space $Y$, an operator \,$T : X \rightarrow Y$ \,is said to be {\it disjointly strictly singular} (DSS) if, for every disjoint sequence of non-null vectors $(x_n)_{n = 1}^{\infty} \subset X$, the restriction  \,$T|_{\overline{\rm span}\{x_n : n\in \mathbb N\}}$ \,is not an isomorphism. Recall that $(x_n)_{n=1}^{\infty}$ is said to be disjoint whenever \,$x_n \wedge x_m = 0$ if $n \ne m$. Every strictly singular operator is known to be DSS, but the converse is not true in general. Recall that a bounded linear operator \,$T : X \to Y$ \,between normed spaces is said to be {\it strictly singular} if it is not bounded below on any infinite-dimensional subspace.

\vskip 3pt

Characterizations of DSS inclusion operators for Orlicz spaces and Lorentz spaces are provided in \cite{Ast, HR}. However, the characterization for Orlicz-Lorentz function spaces has not yet been known explicitly. We mention that the structure of these spaces depends on both Orlicz functions \,$\varphi$ \,and weight functions $w$. Here we will only consider the DSS inclusion operators when the Orlicz function $\varphi$ is fixed.

\subsection{DSS inclusion operators between Orlicz-Lorentz spaces with a fixed Orlicz function $\varphi$}
Now, for a given Orlicz function $\varphi$ that satisfies the $\Delta_2$-condition when $\gamma = \infty$ or the $\Delta_2^{\infty}$-condition for $\gamma < \infty$, we provide a characterization of inclusion operators between Orlicz-Lorentz spaces with different weight functions. We will consider the modular $\rho_{\varphi,w}$ in the following form:
\[
\rho_{\varphi,w}(f) = \int_I W(d_{\varphi\circ f^*}(\lambda))d\lambda,	
\]
where $d_{\varphi \circ f^*}(\lambda)$ is the distribution function of $\varphi\circ f^*$.

The following auxiliary result will be used in the proof of Theorem \ref{Thm Orl-Lor inclusion-2}.
For a set \,$E \subset \R$, we denote, as usual, the characteristic function of \,$E$ \,by \,$\chi_E$.

\begin{Lemma} \label{lem:inverse of W(t)}
	For a given Orlicz function \,$\varphi$ \,and a given weight \,$w$, we have
	$$
	\varphi \left( \frac{1}{\| \chi_{[0,t]} \|_{\varphi ,w}} \right) = \frac{1}{W(t)} \hbox{ \ for all \ } t > 0.
	$$
\end{Lemma}

\begin{proof}
	Let us fix \,$t > 0$. Since \,$\chi_{[0,t]}$ \,is nonincreasing, we have \,$\chi_{[0,t]} = \left( \chi_{[0,t]} \right)^*$. Then, we obtain
	\begin{align*}
		\varphi \left( \frac{1}{\| \chi_{[0,t]} \|_{\varphi ,w}} \right) &= \varphi \left( \frac{1}{\inf \left\{ \var > 0 : \, \int_I \varphi \left( \frac{\chi_{[0,t](s)}}{\var} \right) \cdot w(s) \,ds  \le 1\right\} } \right) \\
		&=\varphi \left( \frac{1}{\inf \left\{ \var > 0 : \, \int_{0}^{t} \varphi (\frac{1}{\var}) \cdot w(s) \,ds  \le 1\right\}  } \right)\\
		&= \varphi \left( \frac{1}{\inf \left\{ \var > 0 : \, \varphi ({1 \over \var}) \cdot W(t)   \le 1\right\}  } \right) \\
		&= \varphi \left( \sup \left\{ \alpha > 0 : \, \varphi (\alpha ) \le \frac{1}{W(t)} \right\}  \right) = \frac{1}{W(t)},
	\end{align*}
	where in the last equality the monotonicity and the continuity of \,$\varphi$ \,have been used.
\end{proof}

Note also that if we replace the weight \,$w$ \,by a multiple \,$M w$ \,$(M > 0)$ \,then, as it is readily seen, we have \,$\|f\|_{\varphi,Mw} \to \infty$ \,as \,$M \to \infty$. Furthermore, if an Orlicz function $\varphi$ satisfies the appropriate $\Delta_2$-condition, one can see, after normalization, that for any \,$c > 0$,\, there exists \,$M > 0$ \,such that  \,$c \cdot \|f\|_{\varphi,w} \le \|f\|_{\varphi,Mw}$ for every $f \in \Lambda_{\varphi,w}$. Indeed, notice that $\rho_{\varphi,Mw}(f) = M \cdot \rho_{\varphi,w}(f)$ for every $f \in \Lambda_{\varphi,w}$ and $M > 0$. Now, since $\varphi$ satisfies the appropriate $\Delta_2$-condition, $\rho_{\varphi,w}\left(\frac{f}{\|f\|_{\varphi,w}}\right) = 1$. So for every $M \geq c$, we have
	\begin{eqnarray*}
		\rho_{\varphi,w}\left(\frac{cf}{\|f\|_{\varphi,Mw}}\right) &=& \rho_{\varphi,w}\left(\frac{cf/\|f\|_{\varphi,w}}{\|f\|_{\varphi,Mw}/\|f\|_{\varphi,w}}\right) \\
		&\leq& \rho_{\varphi,w}\left(\frac{cf/\|f\|_{\varphi,w}}{\rho_{\varphi,Mw}(f/\|f\|_{\varphi,w})}\right)\\
		&=& \rho_{\varphi,w}\left(\frac{cf/\|f\|_{\varphi,w}}{M\rho_{\varphi,w}(f/\|f\|_{\varphi,w})}\right)\\
		&\leq& \frac{c}{M} \rho_{\varphi,w}\left(\frac{f}{\|f\|_{\varphi,w}}\right) = \frac{c}{M} \leq 1.
	\end{eqnarray*}
	Hence, we see that $c \cdot \|f\|_{\varphi,w} \leq \|f\|_{\varphi,Mw}$ for such $M \geq c$. We mention that the appropriate $\Delta_2$-condition cannot be omitted for this approach because there exists a sequence of functions $(f_n)_{n=1}^{\infty} \in S_{\Lambda_{\varphi,w}}$ with pairwise disjoint supports such that $\rho_{\varphi,w}(f_n) \leq \frac{1}{2^n}$ (see \cite[Lemma 2.3]{K}) if $\varphi$ does not satisfy the appropriate $\Delta_2$-condition.

Also, observe that \,$\Lambda_{\varphi,w} = \Lambda_{\varphi,Mw}$ \,when considered only as sets. Then the closed graph theorem tells us that, under their respective topologies, these Banach spaces are isomorphic.

\begin{Theorem}\label{Thm Orl-Lor inclusion-2}
	Let $I=[0,\gamma)$ with $\gamma \leq \infty$, $\varphi$ be an Orlicz function that satisfies the $\Delta_2$-condition for $\gamma = \infty$ ($\Delta_2^{\infty}$-condition for $\gamma < \infty$) and $w_1$ and $w_2$ be weight functions on $I$. We denote \,$W_i(t) = \int_{0}^{t} w_i(s) \,ds$ \,for \,$i=1,2$.
	Then the following statements are equivalent:
	\begin{enumerate}[\rm(i)]
		\item There exists a constant \,$K >0$ \,such that \,$W_2(t) \leq K \cdot W_1(t)$ \,for all \,$t \in I$.
		\item $\Lambda_{\varphi,w_1} \subset \Lambda_{\varphi,w_2}$, the inclusion being continuous.
	\end{enumerate}
\end{Theorem}

\begin{proof}
	(i) $\implies$ (ii): Suppose that \,$W_2(t) \leq K \cdot W_1(t)$ \,$(t \in I)$ \,and let \,$f \in \Lambda_{\varphi,w_1}$.
	Then $\rho_{\varphi,w_1}(kf) < \infty$ for some $k > 0$. Hence we have
	\[
	\rho_{\varphi,w_2}(kf) = \int_I W_2(d_{\varphi \circ kf^*}(\lambda)) \, d\lambda \leq \int_I K \cdot W_1(d_{\varphi \circ kf^*}(\lambda))d\lambda =
	K \, \rho_{\varphi,w_1} (kf) < \infty.
	\]
	Therefore, $f \in \Lambda_{\varphi,w_2}$ and so $\Lambda_{\varphi,w_1} \subset \Lambda_{\varphi,w_2}$.
	
	In order to prove the continuity of the inclusion, without loss of generality, we assume that \,$K > 1$. Then by the convexity of $\varphi$, for a given \,$f \in \Lambda_{\varphi,w_1}$ we have
	\begin{equation*}
		\begin{split}
			\rho_{\varphi,w_2} \left({f \over K \, \|f\|_{\varphi,w_1}} \right) &= \int_I W_2\left(d_{\varphi \circ {f^* \over K \,\|f\|_{\varphi,w_1}}}(\lambda)\right) \, d\lambda \leq \int_I
			K \cdot W_1\left(d_{\varphi \circ {f^* \over K \, \|f\|_{\varphi,w_1}}}(\lambda)\right) \, d\lambda\\
			&= K \cdot \rho_{\varphi,w_1} \left({f \over K \, \|f\|_{\varphi,w_1}} \right)
			=  K \cdot \int_I \varphi \left(\frac{f^*(t)}{K \, \|f\|_{\varphi,w_1}}\right) w_1(t) \,dt \\
			&\le K \cdot {1 \over K} \cdot \int_I \varphi \left(\frac{f^*(t)}{\|f\|_{\varphi,w_1}}\right) w_1(t) \,dt
			= \rho_{\varphi,w_1} \left({f \over \|f\|_{\varphi,w_1}} \right) \le 1.
		\end{split}
	\end{equation*}
	Hence, $\|f\|_{\varphi,w_2} \le K \cdot \|f\|_{\varphi,w_1}$, which yields the continuity.
	
	\vskip 2pt
	
	\noindent (ii) $\implies$ (i): We start from \,$\Lambda_{\varphi,w_1} \subset \Lambda_{\varphi,w_2}$, where the inclusion is continuous.
	For every $t \in I$, let the set \,$E_t := [0,t] \subset I$.
	It follows from the assumption that
	$$\|\chi_{E_t}\|_{\varphi, w_2} \leq c \cdot \|\chi_{E_t}\|_{\varphi,w_1}$$
	for some \,$c > 0$ \,not depending on \,$t$. 
		As observed before the theorem, if $\varphi$ satisfies the appropriate $\Delta_2$-condition, we can choose \,$M > 0$ \,such that  \,$c \cdot \|f\|_{\varphi,w_1} \le \|f\|_{\varphi,Mw_1}$ \,for each \,$f \in \Lambda_{\varphi, w_1}$.
		In particular, we derive
		$$\|\chi_{E_t}\|_{\varphi, w_2} \leq \|\chi_{E_t}\|_{\varphi,Mw_1} \hbox{ \ for all \ } t \in I.$$
		Then, by Lemma \ref{lem:inverse of W(t)} and the fact that $\varphi$ is increasing, we conclude
		$$
		\frac{1}{W_1(t)}= \frac{M}{M \cdot W_1(t)} =  M\cdot \varphi\left(\frac{1}{\|\chi_{E_t}\|_{\varphi,Mw_1}}\right) \leq
		M\cdot \varphi\left(\frac{1}{\|\chi_{E_t}\|_{\varphi,w_2}}\right) = \frac{M}{W_2(t)}.
		$$
		Therefore, there exists $K = M> 0$ such that \,$W_2(t) \leq K \cdot W_1(t)$ \,for all \,$t \in I$, as required.
\end{proof}

If an Orlicz function \,$\varphi$ \,satisfies the $\Delta_2^{\infty}$-condition, it is well-known that the set of simple functions in $\Lambda_{\varphi,w}[0,1]$ is always dense in $\Lambda_{\varphi,w}[0,1]$ (see \cite[Theorem~1.5.5]{BS} and \cite[Theorem 2.4]{K}). Inspired by Astashkin's characterization of DSS inclusion operators between Lorentz spaces in \cite{Ast}, we provide a characterization of these operators between Orlicz-Lorentz spaces when the fixed Orlicz function satisfies the $\Delta_2^{\infty}$-condition.

\begin{Theorem}\label{Thm Orl-Lor inclusion-2-DSS}
	Assume that an Orlicz function \,$\varphi$ \,satisfies the $\Delta_2^{\infty}$-condition and let $w_1$ and $w_2$ be weight functions on $[0,1]$. If \,$\lim_{t\rightarrow 0^+} \frac{W_2(t)}{W_1(t)} = 0$, then the inclusion operator $J: \Lambda_{\varphi,w_1}[0,1] \rightarrow \Lambda_{\varphi,w_2}[0,1]$ is DSS.
\end{Theorem}

\begin{proof}
	If $\lim_{t\rightarrow 0^+} \frac{W_2(t)}{W_1(t)} = 0$, then there exists \,$K > 0$ \,such that \,$W_2(t) \leq K W_1(t)$ \,for all \,$t \in [0,1]$ \,due to the continuity of \,$W_1(t)$ \,and \,$W_2(t)$. Hence, the inclusion operator \,$J: \Lambda_{\varphi,w_1}[0,1] \rightarrow \Lambda_{\varphi,w_2}[0,1]$ \,is well-defined by Theorem \ref{Thm Orl-Lor inclusion-2}. Now, suppose, by way of contradiction, that the inclusion operator \,$J:\Lambda_{\varphi,w_1}[0,1] \rightarrow \Lambda_{\varphi,w_2}[0,1]$ \,is not DSS. Then there exists a sequence \,$(f_n)_{n=1}^{\infty} \subset \Lambda_{\varphi,w_1}[0,1] \setminus \{0\}$ \,of disjointly supported functions $f_n \geq 0$ such that $\|f_n\|_{\varphi,w_1} \leq C\|f_n\|_{\varphi,w_2}$ for some \,$C > 0$ \,independent of \,$n$. From the fact that $\lim_{t\rightarrow 0^+} \frac{W_2(t)}{W_1(t)} = 0$, for every $\var \in (0,1)$, there exists $s \in (0,1]$ such that $W_2(t) \leq \var W_1(t)$ for every $t < s$. This implies that we have $\rho_{\varphi,w_2}(f) \leq \var \rho_{\varphi,w_1}(f) \leq \rho_{\varphi,w_1}(f)$ for every function $f \in \Lambda_{\varphi, w_1}$ with $m (\supp f) < s$, and so \,$\var \|f\|_{\varphi,w_2} \leq \|f\|_{\varphi,w_2} \leq \|f\|_{\varphi,w_1}$ \,for such functions.
	
	Now, from the fact that $f_n$'s are disjointly supported on $[0,1]$, we can choose $N \in \mathbb{N}$ such that $m(\supp f_n) < s$ for every $n \geq N$. As mentioned before, if $\varphi$ satisfies the $\Delta_2^{\infty}$-condition, then $\Lambda_{\varphi,w_i} =(\Lambda_{\varphi,w_i})_a = (\Lambda_{\varphi,w_i})_b$, $i = 1,2$.
	Hence there exists a sequence $(g_n)_{n=N}^{\infty}$ of simple functions such that $\supp g_n \subset \supp f_n$ and
	\[
	\max\left\{\left\|\frac{f_n}{\|f_n\|_{\varphi,w_1}} - g_n\right\|_{\varphi,w_1}, \left\|\frac{f_n}{\|f_n\|_{\varphi,w_2}} - g_n\right\|_{\varphi,w_2}\right\} \leq \var.
	\]
	
	
	From the fact that $1 - \var \leq \|g_n\|_{\varphi,w_i} \leq 1 + \var$ for $i = 1, 2$, we have
	\begin{eqnarray*}
		\left\|\frac{f_n}{\|f_n\|_{\varphi,w_i}} - \frac{g_n}{\|g_n\|_{\varphi,w_i}}\right\|_{\varphi,w_i} &\leq& \left\|\frac{f_n}{\|f_n\|_{\varphi,w_i}} - g_n\right\|_{\varphi,w_i}  + \left\|g_n- \frac{g_n}{\|g_n\|_{\varphi,w_i}}\right\|_{\varphi,w_i}\\
		&\leq& \var + \|g_n\|_{\varphi,w_i}\left|1 - \frac{1}{\|g_n\|_{\varphi, w_i}}\right| \leq 2 \var.
	\end{eqnarray*}
	Let $f_n' = \frac{f_n}{\|f_n\|_{\varphi,w_2}}$ and $g_n' = \frac{g_n}{\|g_n\|_{\varphi,w_2}}$. Since $\varphi$ satisfies the $\Delta_2^{\infty}$-condition, $\rho_{\varphi,w_2}(g_n') = \|g_n'\|_{\varphi,w_2}$. Thus, we obtain
	\begin{align*}
		\|f_n'\|_{\varphi,w_2} &\leq \frac{1}{1 - 2\var}\|g_n'\|_{\varphi,w_2} = \frac{1}{1 - 2\var}\rho_{\varphi,w_2}(g_n') \leq \frac{\var}{1 - 2\var} \rho_{\varphi,w_1}(g_n')\\
		&\leq \frac{\var\|g_n\|_{\varphi,w_1}}{(1 - 2\var)\|g_n\|_{\varphi,w_2}}
		\leq \frac{\var}{(1 - 2\var)^2}\|g_n\|_{\varphi,w_1} \leq  \frac{(1 + \var)\var}{(1-2\var)^2} \leq \frac{(1 + \var)\var}{(1-2\var)^2}\|f_n'\|_{\varphi,w_1}.
	\end{align*}
	
	If we choose \,$\var > 0$ \,such that \,$C < \frac{(1- 2\var)^2}{\var(1 + \var)}$, then we get \,$C\|f_n\|_{\varphi,w_2} < \|f_n\|_{\varphi,w_1}$
	\,for all \,$n \ge N$, which contradicts our assumption. The proof is finished.
\end{proof}

\begin{Theorem}\label{Thm OL DSS characterization}
	Assume that the Orlicz function \,$\varphi$ \,satisfies the $\Delta_2^{\infty}$-condition and that \,$w_1$ \,and \,$w_2$ \,are weight functions on $[0,1]$.
	Then the following statements are equivalent:
	\begin{enumerate}[\rm(i)]
		\item $\lim_{t\rightarrow 0^+}\frac{W_2(t)}{W_1(t)} = 0$.
		\item The inclusion operator $J: \Lambda_{\varphi,w_1}[0,1] \rightarrow \Lambda_{\varphi,w_2}[0,1]$ is DSS.
		\item There exist no sequence of nonzero disjoint functions \,$(f_n)_{n=1}^{\infty}$ \,and no constant \,$K > 0$ \,such that
		\,$\|f_n\|_{\varphi, w_1} \leq K \|f_n\|_{\varphi,w_2}$ \,for all \,$n \in \mathbb{N}$.
	\end{enumerate}
\end{Theorem}

\begin{proof}
	The implication (i) $\implies$ (ii) is a direct consequence of Theorem \ref{Thm Orl-Lor inclusion-2-DSS}, while (ii) $\implies$ (iii) comes from the definition. So it suffices to show (iii) $\implies$ (i).
	
	Assume, by way of contradiction, that $\liminf_{t\rightarrow 0^+} \frac{W_2(t)}{W_1(t)} > 0$. Then there exists a sequence of positive real numbers \,$(t_n)_{n=1}^{\infty}$ \,with \,$\sum_{n=1}^{\infty}t_n \leq 1$ \,as well as a constant \,$K > 0$ such that \,$W_1(t_n) \leq K \cdot W_2(t_n)$ \,for all \,$n \in \mathbb{N}$.
	Without loss of generality, we can assume \,$K > 1$. Also, we can select a collection
	\,$(J_n)_{n=1}^\infty$ \,of pairwise disjoint intervals in \,$[0,1]$ \,with \,$m(J_n) = t_n$  $(n \ge 1)$.
	Define \,$f_n = \chi_{J_n}$ $(n \in \mathbb{N})$, and observe that these functions have disjoint supports. It is easy to see that \,$f_n^* = \chi_{[0,t_n]}$. Now, fix \,$n \in \mathbb{N}$ \,and note that
	\begin{equation*}
		\begin{split}
			\rho_{\varphi,w_1} \left({f_n \over K \, \|f\|_{\varphi,w_2}} \right) &= \int_0^{1} \varphi \left({f_n^*(t) \over K \, \|f\|_{\varphi,w_2}} \right) \cdot w_1(t) \,dt
			= \int_0^{t_n} \varphi \left({1 \over K \, \|f\|_{\varphi,w_2}} \right) \cdot w_1(t) \,dt \\
			&= \varphi \left({1 \over K \, \|f\|_{\varphi,w_2}} \right) \cdot \int_0^{t_n}  w_1(t) \,dt = \varphi \left({1 \over K \, \|f\|_{\varphi,w_2}} \right) \cdot W_1(t_n) \\
			&\le K \cdot \varphi \left({1 \over K \, \|f\|_{\varphi,w_2}} \right) \cdot W_2(t_n) \leq \varphi \left({1 \over \|f\|_{\varphi,w_2}} \right) \cdot W_2(t_n)\\
			&= \varphi \left({1 \over \|f\|_{\varphi,w_2}} \right) \cdot \int_0^{t_n}  w_2(t) \,dt  = \int_0^{t_n} \varphi \left({1 \over \|f\|_{\varphi,w_2}} \right) \cdot w_2(t) \,dt \\
			&=  \int_0^{1} \varphi \left({f_n^*(t) \over \|f\|_{\varphi,w_2}} \right) \cdot w_2(t) \,dt = \rho_{\varphi,w_2} \left({f_n \over \|f\|_{\varphi,w_2}} \right) = 1,
		\end{split}
	\end{equation*}
	Hence, $\|f_n\|_{\varphi,w_1} \le K \cdot \|f_n\|_{\varphi,w_2}$ \,for all \,$n \in \mathbb{N}$. This is the desired contradiction.
\end{proof}

For weight functions \,$w$ \,and \,$v$, we denote \,$w \ll v$\, if \,$\displaystyle \lim_{t\rightarrow 0^+} \frac{W(t)}{V(t)} = 0$,
where \,$\displaystyle W(t) = \int_{0}^{t} w$ \,and \,$\displaystyle V(t) = \int_{0}^{t} v$ $(t > 0)$.

From the previous theorem and Corollary \ref{lem:inclusion},
we can derive that certain subsets of Orlicz-Lorentz spaces are empty:

\begin{Theorem} \label{Thm empty difference}
	Assume that \,$\varphi$ \,is an Orlicz function satisfying the $\Delta_2^{\infty}$-condition and that \,$w$ \,is a weight function on $[0,1]$.
	Then \,$\Lambda_{\varphi,w}[0,1] \setminus \bigcup_{w \ll v} \Lambda_{\varphi,v}[0,1]$ \,is empty.
\end{Theorem}

\begin{proof}
	We claim that for a given weight function $w$ and $f \in \Lambda_{\varphi,w}[0,1]$, there exists a weight function $v$ such that $w \ll v$ and $f \in \Lambda_{\varphi,v}[0,1]$. From the fact that $\varphi(f^*)w \in L_1[0,1]$, we can choose \,$g$ \,nonnegative and decreasing such that \,$\lim_{t\rightarrow 0^+} g(t) = \infty$ \,and
	\begin{equation}\label{eq11}
		\int_{0}^{1} \varphi(f^*)\cdot g\cdot w < \infty.
	\end{equation}
	
	Now, define the function $v(s) := g(s)w(s)$ on $[0,1]$. The function $v$ is decreasing, positive and locally integrable, so it is a weight function.
	We can also see that $f \in \Lambda_{\varphi,v}$ from (\ref{eq11}). Furthermore, we have:
	\[
	0 \le \limsup_{t \rightarrow 0^+} \frac{W(t)}{V(t)} = \limsup_{t \rightarrow 0^+} \frac{W(t)}{\int_0^t g(s)w(s) \, ds} \leq \limsup_{t \rightarrow 0^+} \frac{W(t)}{g(t)W(t)} = \lim_{t \rightarrow 0^+} \frac{1}{g(t)} = 0
	\]
	and so \,$w \ll v$.
	Since each element of $f \in \Lambda_{\varphi,w}[0,1]$ belongs to $\Lambda_{\varphi,v}[0,1]$ for some weight $v$ satisfying $w \ll v$, we see that $\Lambda_{\varphi,w}[0,1] \subset \bigcup_{w \ll v} \Lambda_{\varphi,v}[0,1]$. Moreover, in view of Theorems~\ref{Thm Orl-Lor inclusion-2} and~\ref{Thm Orl-Lor inclusion-2-DSS}, the inclusion $J_v:\Lambda_{\varphi,v} \rightarrow \Lambda_{\varphi,w}$ is well-defined for each $v$ satisfying $w \ll v$, so $\bigcup_{w \ll v} \Lambda_{\varphi,v}[0,1] \subset \Lambda_{\varphi,w}[0,1]$. This shows that $\Lambda_{\varphi,w}[0,1] \setminus \bigcup_{w \ll v} \Lambda_{\varphi,v}[0,1]$ is empty.
\end{proof}

As a direct consequence, we have the analogous statement for the Lorentz space $\Lambda_{p,w}$ where $1 \leq p < \infty$.

\begin{Corollary}\label{cor:Lorentz-empty}
	Let \,$w$ be a weight function on $[0,1]$ and $1 \leq p < \infty$. Then the set \,$\Lambda_{p,w}[0,1] \setminus \bigcup_{w \ll v} \Lambda_{p,v}[0,1]$ is empty.
\end{Corollary}

\subsection{Inclusion operators between Orlicz-Lorentz spaces with a fixed weight $w$} In this subsection, we provide a characterization of inclusion operators between Orlicz-Lorentz spaces. This fact was provided in \cite[Theorem 1.4]{K} with hints for the proof. Here, we present a different proof inspired by \cite[Theorem 5.1.3]{RR}.

\begin{Theorem} \label{Thm Orl-Lor inclusion}
	Let $\varphi_1$ and $\varphi_2$ be Orlicz functions and let $w$ be a weight function on $I = [0, \gamma)$ with $\gamma \leq \infty$. The following statements are equivalent:
	\begin{enumerate}[\rm(i)]
		\item $\varphi_1 \prec \varphi_2$ {\rm (}resp.~$\varphi_1 \prec_{\infty} \varphi_2${\rm )} \,for \,$\gamma = \infty$ \,{\rm (}resp.~$\gamma < \infty${\rm ).}
		\item $\Lambda_{\varphi_2,w} \subset \Lambda_{\varphi_1,w}$, the inclusion being continuous.
	\end{enumerate}
\end{Theorem}

\begin{proof}
	(i) $\implies$ (ii): First, assume that $\varphi_1 \prec \varphi_2$ and let $f \in \Lambda_{\varphi_2,w}$. Then there exists $b > 0$ such that $\varphi_1(u) \leq \varphi_2(bu)$ for all $u \geq 0$. Since Orlicz-Lorentz spaces are rearrangement invariant and $\rho_{\varphi_2,w}\left(\frac{f}{\|f\|_{\varphi_2,w}}\right)\leq 1$, we have
	\[
	\rho_{\varphi_1, w}\left(\frac{f}{b\|f\|_{\varphi_2,w}}\right) = \int_I \varphi_1\left(\frac{f^*}{b \|f\|_{\varphi_2, w}}\right)w \leq \int_I \varphi_2\left(\frac{f^*}{ \|f\|_{\varphi_2, w}}\right)w \leq 1.
	\]
	Hence, $\|f\|_{\varphi_1,w} \leq b \cdot \|f\|_{\varphi_2,w}$ \,and so \,$\Lambda_{\varphi_2,w} \subset \Lambda_{\varphi_1,w}$, the inclusion being continuous.
	
	\vskip 3pt
	
	Now, assume that $\varphi_1 \prec_{\infty} \varphi_2$ and let $f \in \Lambda_{\varphi_2,w}$ such that $\|f\|_{\varphi_2,w} > 0$. Then there exists $b > 0$ and $u_0 > 0$ such that $\varphi_1(u) \leq \varphi_2 (bu)$ for all $u \geq u_0$. Now, define a set $E = \{t \in I: f^*(t) \leq u_0 b \|f\|_{\varphi_2, w}\}$. As a matter of fact, the set \,$E$ \,is the interval $[\gamma - m(E), \gamma)$.  Let $M = \varphi_1(u_0)(W(\gamma) - W(\gamma - m(E))) + 1$. Since Orlicz functions are convex and $\rho_{\varphi_2,w}\left(\frac{f}{\|f\|_{\varphi_2,w}}\right) \leq 1$, we obtain that
	\begin{align*}
		\rho_{\varphi_1, w}\left(\frac{f}{Mb\|f\|_{\varphi_2,w}}\right) &= \rho_{\varphi_1, w}\left(\frac{f^*}{Mb\|f\|_{\varphi_2,w}}\right)= \int_I \varphi_1 \left(\frac{f^*}{Mb\|f\|_{\varphi_2,w}}\right)w \\
		&\leq \frac{1}{M}\left(\int_{\gamma - m(E)}^{\gamma}\varphi_1\left(\frac{f^*}{b\|f\|_{\varphi_2,w}}\right)w + \int_I\varphi_1\left(\frac{f^*}{b\|f\|_{\varphi_2,w}}\right)w\right)\\
		&\leq \frac{1}{M}\left(\varphi_1(u_0)(W(\gamma) - W(\gamma - m(E))) + \int_I\varphi_2\left(\frac{f^*}{\|f\|_{\varphi_2,w}}\right)w\right)\\
		&\leq \frac{1}{M} \cdot M = 1.
	\end{align*}
	Hence, $\|f\|_{\varphi_1,w} \leq b \cdot \|f\|_{\varphi_2,w}$ and so \,$\Lambda_{\varphi_2,w} \subset \Lambda_{\varphi_1,w}$ \,in this case, the inclusion being continuous.
	
	\vskip 3pt
	
	\noindent (ii) $\implies$ (i): Since we can always find a finite measurable subset in $I$, showing the claim for the case $\gamma < \infty$ is enough. Let \,$E$ \,be a measurable subset of $I$ such that $0 < m(E) < \infty$. Suppose that $\varphi_1 \not\prec_\infty \varphi_2$. Then we can find an increasing sequence of real numbers $(a_n)_{n=1}^{\infty}\uparrow \infty$ such that $\varphi_1(a_n) > \varphi_2(2^n n^2a_n)$ for every $n \in \mathbb{N}$. Then by the convexity of the Orlicz function $\varphi_2$, we get that
	\begin{equation} \label{eq12}
		\varphi_1(a_n) > \varphi_2(2^n n^2 a_n) \geq 2^n \varphi_2 (n^2 a_n) \hbox{ \ for all } \, n \in \N. 	
	\end{equation}

	Passing to a subsequence if necessary, the sequence $(a_n)_{n=1}^{\infty}$ can be chosen such that
	\[
	\sum_{n=1}^{\infty} \frac{1}{2^n \varphi_2(n^2a_n)} \le \int_0^{m(E)} w. 
	\]
	Let \,$t_0$ \,be a positive real number such that
	$$
	\sum_{n=1}^{\infty} \frac{1}{2^n \varphi_2(n^2a_n)} = \int_0^{t_0} w.
	$$
	Then there exists a sequence \,$(t_n)_{n=1}^{\infty}\downarrow 0$ \,such that
	\[
	\int_{t_n}^{t_{n-1}} w = \frac{1}{2^n\varphi_2(n^2 a_n)} \hbox{ \ for all \ } n \in \N.
	\]
	
	Now, define the function
	\[
	f := \sum_{n=1}^{\infty} n a_n \chi_{[t_n, t_{n-1}]}.
	\]
	Since \,$f$ \,is nonincreasing, we have \,$f = f^*$, and so
	\[
	\rho_{\varphi_2, w}(f) = \sum_{n=1}^{\infty} \varphi_2(na_n) \int_{t_{n}}^{t_{n-1}} w = \sum_{n=1}^{\infty} \frac{\varphi_2(na_n)}{2^n \varphi_2(n^2 a_n)} < \sum_{n=1}^{\infty} \frac{1}{2^n} = 1.
	\]
	Hence, $f \in \Lambda_{\varphi_2,w}$. On the other hand, choose \,$n_0 \in \N$ \,such that \, $n_0 \var > 1$. Then from \eqref{eq12} we obtain
	\begin{align*}
		\rho_{\varphi_1,w}(\var f) &\geq \sum_{n=n_0 + 1}^{\infty} \varphi_1 (\var na_n) \int_{t_{n}}^{t_{n-1}} w\\
		&\geq \sum_{n=n_0 + 1}^{\infty} \varphi_1 (a_n) \int_{t_{n}}^{t_{n-1}} w\\
		& \geq  \sum_{n=n_0 + 1}^{\infty} \frac{2^n\varphi_2(n^2 a_n)}{2^n\varphi_2(n^2 a_n)} = \infty.
	\end{align*}
	Therefore, $f \notin \Lambda_{\varphi_1,w}$ and so $\Lambda_{\varphi_2,w} \not\subset \Lambda_{\varphi_1,w}$.
\end{proof}

{\color{blue}

}


\section{Spaceability and residuality on Orlicz-Lorentz function spaces}
In this section, we shall analyze the spaceability and residuality of various nonlinear subsets of Orlicz-Lorentz function spaces.

Before presenting our first result, we mention that the measure space for Orlicz spaces $L_{\varphi}(\mu)$ in \cite{AM} is an abstract measure space with a positive measure $\mu$. Hence, in this case, we need to consider the conditions $(\alpha)$ and $(\beta)$ defined by
\begin{enumerate}
	\item[($\alpha$)] $\inf\{\mu (A) : A \in \Sigma,\ \mu (A) > 0\} = 0$.
	\item[($\beta$)] $\sup\{\mu (A) : A \in \Sigma,\ \mu (A) < \infty\} = \infty$.
\end{enumerate}
However, we only consider the Lebesgue measure on $I$ for Orlicz-Lorentz spaces, so the condition $(\alpha)$ is automatically given because the measure is nonatomic. The condition $(\beta)$ is satisfied only in the case of $I = [0, \infty)$.

\vskip 3pt

The following result holds for both bounded and unbounded \,$I = [0, \gamma)$, $\gamma \leq \infty$.

\begin{Theorem}\label{th:OLspaceable1}
Assume that \,$w$ \,is a weight function on \,$I=[0,\gamma)$ with $\gamma \leq \infty$. Let $\varphi$, $\varphi_n$ $(n \geq 1)$ be Orlicz functions such that each \,$\varphi_n$ \,satisfies the $\Delta_\varphi(\infty)$-condition. Then the set \,$(\Lambda_{\varphi,w})_a \setminus \bigcup_{n=1}^\infty \Lambda_{\varphi_n,w}$ is $\mathfrak c$-spaceable in \,$(\Lambda_{\varphi,w})_a$.
Moreover, if $\varphi$ satisfies the $\Delta_2$-condition when $\gamma = \infty$ (resp. $\Delta_2^\infty$-condition when $\gamma < \infty$), then \,$\Lambda_{\varphi,w} \setminus \bigcup_{n=1}^\infty \Lambda_{\varphi_n,w}$ is maximal-spaceable in \,$\Lambda_{\varphi,w}$.
\end{Theorem}

\begin{proof}
We only show this for the case $\gamma = \infty$ because the finite case is similar. Let
\,$X := (\Lambda_{\varphi,w})_a$ \,and \,$B := \bigcup_{n=1}^\infty \Lambda_{\varphi_n,w}$.
As we mentioned, conditions (i) and (ii) in Theorem \ref{th:pcs} are given immediately. Since the order-continuous subspace of a Banach function space is a cone, the condition (iii) in Theorem \ref{th:pcs} holds. From the fact that the order-continuous subspace of a Banach function space is a closed order-ideal, we get the condition (iv) in Theorem \ref{th:pcs}. Hence, it suffices to show that if $E \in \Sigma$ with $m(E) > 0$, then there exists a function \,$f$ \,in $X \setminus B$ whose support is $E$. Indeed, by choosing pairwise disjoint measurable sets \,$S_j \subset I$ \,($j \in \N$) \,with \,$m(S_j) > 0$ \,would do the job of fulfilling (v) in Theorem 2.4.

\vskip 3pt
	
Since $\varphi_n$ satisfies $\Delta_{\varphi}(\infty)$ for all $n \in \mathbb{N}$, there exists an increasing sequence of real numbers $(u_{n,k})_{k=1}^{\infty} \uparrow \infty$ such that
\begin{equation}
	\varphi_n(u_{n,k}) \geq 2^k \varphi(k^2 u_{n,k}).
\end{equation}

\vskip 3pt

Let $E \in \Sigma$ such that $m(E) = t_0 < \infty$ and let $(E_n)_{n=1}^{\infty}$ be a partition of $E$ such that $m(E_n) = \frac{t_0}{2^n}$. By passing to a subsequence if necessary, we can choose the sequence $(u_{n,k})_{k=1}^{\infty}$ such that
\begin{equation}
	\sum_{k=1}^{\infty}\frac{1}{2^k \varphi(k^2 u_{n,k})} = \int_0^{t_0/2^n}w.
\end{equation}
Then, we can find a decreasing sequence $(t_{n,k})_{k=1}^{\infty} \downarrow 0$ of real numbers such that, for every $k \in \mathbb{N}$, it holds that
\begin{equation}\label{eq:wpieces1}
	\int_{t_{n,k+1}}^{t_{n,k}}w = \frac{1}{2^k\varphi(k^2u_{n,k})}.
\end{equation}
Furthermore, we can construct a partition $(E_{n,k})_{k=1}^{\infty}$ of $E_n$ such that $m(E_{n,k}) = t_{n, k} - t_{n,k+1}$ by using the nonatomicity of the Lebesgue measure.
\vskip 3pt

Now, define the function \,$f_n : I \to \R$ \,by
\[	
f_n = \sum_{k=1}^{\infty} k u_{n,k} \chi_{E_{n,k}}.
\]
We see that $f_n^* = \sum_{k=1}^{\infty}ku_{n,k} \chi_{[t_{n,k+1}, t_{n,k})}$ for every $n \in \mathbb{N}$.
Let
\[
f := \sum_{n=1}^{\infty} \frac{f_n}{2^n \|f_n\|}.
\]
For every $\lambda > 0$, by the orthogonal subadditivity and convexity of $\rho_{\varphi,w}$ and (\ref{eq:wpieces1}) there exists $k_n > \frac{\lambda}{\|f_n\|}$ such that
\begin{eqnarray*}
	\rho_{\varphi,w}(\lambda f) &\leq& \sum_{n=1}^{\infty}\frac{1}{2^n}\cdot \rho_{\varphi,w}\left(\frac{\lambda f_n}{\|f_n\|}\right) = \sum_{n=1}^{\infty}\frac{1}{2^n}\cdot\left(\sum_{k=1}^{\infty} \varphi\left(\frac{\lambda k u_{n,k}}{\|f_n\|}\right)\int_{t_{n,k+1}}^{t_{n,k}}w\right)\\
	&=&\sum_{n=1}^{\infty}\frac{1}{2^n}\cdot\left(\sum_{k=1}^{k_n}\varphi\left(\frac{\lambda k u_{n,k}}{\|f_n\|}\right) \int_{t_{n,k+1}}^{t_{n,k}}w + \sum_{k= k_n + 1}^{\infty}\varphi\left(\frac{\lambda k u_{n,k}}{\|f_n\|}\right) \int_{t_{n,k+1}}^{t_{n,k}}w\right)\\
	&\leq& \sum_{n=1}^{\infty}\frac{1}{2^n}\cdot\left( \varphi(k_n^2 u_{n,k_n}) \sum_{k=1}^{k_n}\int_{t_{n,k+1}}^{t_{n,k}}w +  \sum_{k= k_n + 1}^{\infty}\varphi(k^2 u_{n,k}) \int_{t_{n,k+1}}^{t_{n,k}}w\right)\\
	&\leq&
	\sum_{n=1}^{\infty}\frac{1}{2^n}\cdot\left( \varphi(k_n^2 u_{n,k_n}) \sum_{k=1}^{k_n}\int_{t_{n,k+1}}^{t_{n,k}}w +  \sum_{k= k_n + 1}^{\infty}\frac{1}{2^k}\right).
\end{eqnarray*}
Moreover, notice from (\ref{eq:wpieces1}) that $\int_{t_{n,k_n+1}}^{t_{n,k_n}}w \geq \int_{t_{n,k+1}}^{t_{n,k}}w$ for every $k < k_n$. Hence, we obtain
\begin{eqnarray*}
	\rho_{\varphi,w}(\lambda f) &\leq& \sum_{n=1}^{\infty}\frac{1}{2^n}\cdot\left( \varphi(k_n^2 u_{n,k_n}) \sum_{k = 1}^{k_n}\int_{t_{n,k+1}}^{t_{n,k}}w +  \sum_{k= k_n + 1}^{\infty}\frac{1}{2^k}\right)\\
	&\leq&\sum_{n=1}^{\infty}\frac{1}{2^n}\cdot\left(\frac{k_n}{2^{k_n}} + \sum_{k= k_n + 1}^{\infty}\frac{1}{2^k}\right) \leq 1 < \infty.
\end{eqnarray*}
This shows that $f \in X = (\Lambda_{\varphi,w})_a$. On the other hand, for every $\varepsilon > 0$ and for each $n \in \mathbb{N}$, if $\frac{\varepsilon k_n}{2^n \|f_n\|} > 1$ for some $k_n \in \mathbb{N}$, we have by the monotonicity of $\rho_{\varphi_n,w}$ that
\[\rho_{\varphi_n,w}(\varepsilon f)\geq \rho_{\varphi_n,w}\left(\frac{\varepsilon f_n}{2^n \|f_n\|}\right) \geq \sum_{k = k_n + 1}^{\infty} \varphi_n(u_{n,k}) \int_{t_{n,k+1}}^{t_{n,k}}w\\
\geq \sum_{k= k_n + 1}^{\infty} \frac{2^k\varphi(k^2 u_{n,k})}{2^k\varphi(k^2 u_{n,k})} = \infty.
\]
Thus, $f \in X \setminus B = (\Lambda_{\varphi,w})_a \setminus \bigcup_{n=1}^{\infty}\Lambda_{\varphi_n,w}$.


For the second part, assume that $\gamma=\infty$ and $\varphi$ satisfies the $\Delta_2$-condition (the case when $\gamma < \infty$ and $\varphi$ satisfies the $\Delta_2^\infty$-condition can be done in a similar way).
Since $\Lambda_{\varphi,w} = (\Lambda_{\varphi,w})_a $ is an infinite-dimensional separable Banach space, it is well-known that $\dim (\Lambda_{\varphi,w}) = \mathfrak c$.
Hence, by the first part, \,$\Lambda_{\varphi,w} \setminus \bigcup_{n=1}^{\infty}\Lambda_{\varphi_n,w}$\, is maximal-spaceable in \,$\Lambda_{\varphi,w}$.
The proof is finished.
\end{proof}

As a consequence we obtain the following corollary, but in order to prove it let us prove first the following lemma.

\begin{Lemma}\label{lem:strongorder1}
	Let $I = [0, \gamma)$ with $\gamma \leq \infty$, $\varphi$ and $\varphi_n$ $(n \geq 1)$ be Orlicz functions, and $w$ be a weight function on $I$.
	If $\varphi_{n+1} \prec \varphi_n$ {\rm (}resp.~$\varphi_{n+1} \prec_{\infty} \varphi_n${\rm )} \,for \,$\gamma = \infty$ \,{\rm (}resp.~$\gamma < \infty${\rm )} for every $n\in \mathbb N$, then $\bigcup_{n=1}^\infty \Lambda_{\varphi_n,w}$ is a nontrivial vector subspace of \,$(\Lambda_{\varphi,w})_a$.
\end{Lemma}

\begin{proof}
	 Take $f,g \in \bigcup_{n=1}^\infty \Lambda_{\varphi_n,w}$ and $\alpha,\beta \in \R$.
	 There exist $r,s \in \mathbb N$ such that $f \in \Lambda_{\varphi_r,w}$ and $g \in \Lambda_{\varphi_s,w}$.
	Now take $k=\max\{ r,s \}$.
	Then, since $\varphi_{n+1} \prec \varphi_n$ {\rm (}resp.~$\varphi_{n+1} \prec_{\infty} \varphi_n${\rm )} for $\gamma = \infty$ {\rm (}resp.~$\gamma < \infty${\rm )}, by Theorem~\ref{Thm Orl-Lor inclusion}, we have that $f,g \in \Lambda_{\varphi_k,w}$.
	Hence, $\alpha f + \beta g \in \Lambda_{\varphi_k,w} \subset \bigcup_{n=1}^\infty \Lambda_{\varphi_n,w}$, as needed.
\end{proof}

\begin{corollary}\label{co:OLspaceable1}
	Let $\alpha \geq \aleph_0$ be a cardinal number, $I = [0, \gamma)$ with $\gamma \leq \infty$, and assume that \,$w$ \,is a weight function on $I$.
	Let $\varphi$, $\varphi_n$ $(n \geq 1)$ be Orlicz functions such that each \,$\varphi_n$ \,satisfies the $\Delta_\varphi(\infty)$-condition and \,$\varphi_{n+1} \prec \varphi_n$ {\rm (}resp.~$\varphi_{n+1} \prec_{\infty} \varphi_n${\rm )} \,for \,$\gamma = \infty$ \,{\rm (}resp.~$\gamma < \infty${\rm )} for every $n\in \mathbb N$.
		\begin{itemize}
			\item[\textup{(i)}] If $\aleph_0 \leq \alpha\leq \mathfrak c$ or $\alpha > \emph{card} ((\Lambda_{\varphi,w})_a)$, then $(\Lambda_{\varphi,w})_a \setminus \bigcup_{n=1}^\infty \Lambda_{\varphi_n,w}$ is not $(\alpha,\beta)$-spaceable in $(\Lambda_{\varphi,w})_a$, regardless of the cardinal number $\beta \geq \alpha$;
			\item[(ii)] If \,$\varphi$ satisfies the $\Delta_2$-condition when $\gamma = \infty$ {\rm (}resp. the $\Delta_2^\infty$-condition when $\gamma < \infty${\rm )}, then $\Lambda_{\varphi,w} \setminus \bigcup_{n=1}^\infty \Lambda_{\varphi_n,w}$ is not $(\alpha,\beta)$-spaceable in $\Lambda_{\varphi,w}$, regardless of the cardinal number $\beta \geq \alpha$.
		\end{itemize}
\end{corollary} 

\begin{proof}
	Fix $\beta \geq \alpha$ a cardinal number.
	
	Let us prove first item (i).
	If $\alpha > \text{card} ((\Lambda_{\varphi,w})_a)$, then $(\Lambda_{\varphi,w})_a \setminus \bigcup_{n=1}^\infty \Lambda_{\varphi_n,w}$ is not $\alpha$-lineable.
	Assume now that $\alpha \leq \mathfrak c$.
	Since $(\Lambda_{\varphi,w})_a$ is a Banach space, $(\Lambda_{\varphi,w})_a \setminus \bigcup_{n=1}^\infty \Lambda_{\varphi_n,w}$ is $\alpha$-lineable by Theorem~\ref{th:OLspaceable1}, and \,$\bigcup_{n=1}^\infty \Lambda_{\varphi_n,w}$ is a nontrivial vector subspace of \,$(\Lambda_{\varphi,w})_a$ by Lemma~\ref{lem:strongorder1}. Hence, the result follows by \cite[Corollary~2.5]{FPRR}.
	
	Now let us prove item (ii).
	Assume that $\gamma = \infty$ and $\varphi$ satisfies the $\Delta_2$-condition (the case when $\gamma < \infty$ and $\varphi$ satisfies the $\Delta_2^\infty$-condition can be done analogously).
	If $\alpha > \mathfrak c$, then $(\Lambda_{\varphi,w})_a \setminus \bigcup_{n=1}^\infty \Lambda_{\varphi_n,w}$ is not $\alpha$-lineable since $\Lambda_{\varphi,w} = (\Lambda_{\varphi,w})_a$ has cardinality $\mathfrak c$ under the $\Delta_2$-condition.
	Assuming that $\alpha \leq \mathfrak c$ and applying the same arguments as in the proof of item (i), the latter proves item (ii).
\end{proof}

Assuming $I = [0, \infty)$, we also obtain the analogous statement with respect to those $\varphi_n$'s satisfying the $\Delta_{\varphi}(0)$-condition.

\begin{Theorem}\label{th:OLspaceable2}
Let \,$I = [0, \infty)$ \,and assume that \,$w$ \,is a weight function on \,$I$. Let \,$\varphi, \varphi_n$
$(n \geq 1)$ \,be Orlicz functions such that each \,$\varphi_n$ \,satisfies the $\Delta_{\varphi}(0)$-condition. Then the set \,$(\Lambda_{\varphi,w})_a \setminus \bigcup_{n=1}^\infty \Lambda_{\varphi_n,w}$ \,is $\mathfrak c$-spaceable in \,$(\Lambda_{\varphi,w})_a$.
Moreover, if $\varphi$ satisfies the $\Delta_2$-condition, then \,$\Lambda_{\varphi,w} \setminus \bigcup_{n=1}^\infty \Lambda_{\varphi_n,w}$ is maximal-spaceable in \,$\Lambda_{\varphi,w}$.
\end{Theorem}

\begin{proof}
Let \,$X := (\Lambda_{\varphi,w})_a$ \,and \,$B := \bigcup_{n=1}^\infty \Lambda_{\varphi_n,w}$. In view of Theorem \ref{th:pcs}, and similarly to the beginning of the proof of Theorem \ref{th:OLspaceable1}, it suffices to show that for \,$E \in \Sigma$ with $m(E) = \infty$, there exists a function in $X \setminus B$ whose support is $E$.

\vskip 3pt

Since $\varphi_n$ satisfies $\Delta_{\varphi}(0)$ for all $n \in \mathbb{N}$, there exists a decreasing sequence of real numbers $(u_{n,k})_{k=1}^{\infty} \downarrow 0$ such that
\begin{equation}\label{ineq:deltaphi0}
	\varphi_n\left(\frac{u_{n,k}}{k}\right) \geq 2^k\varphi(k u_{n,k}).
\end{equation}

Let $(E_n)_{n=1}^{\infty}$ be a partition of measurable subsets of $E$ such that \,$m(E_n) = \infty$. Furthermore, by the $\sigma$-finiteness of the Lebesgue measure, let \,$(E_{n,k})_{k=1}^{\infty}$\, be a partition of measurable subsets of $E_n$ such that \,$m(E_{n,k}) < \infty$\, for every \,$k \in \mathbb{N}$\, and \,$m(E_{n,k}) \uparrow \infty$\, as \,$k \rightarrow \infty$. Then, by passing to subsequences if necessary, we can choose $(u_{n,k})_{k=1}^{\infty}\downarrow 0$ such that
	
\begin{equation}\label{eq:wpieces2}
	\frac{1}{2^k \varphi(k u_{n,k})} = \int_{\sum_{i=1}^{k}m(E_{n,i})}^{\sum_{i=1}^{k+1}m(E_{n,i})}w
\end{equation}
for every $k \in \mathbb{N}$. We see that $\frac{1}{2^k \varphi(k u_{n,k})} \uparrow \infty$ as $k \rightarrow \infty$. For simplicity, we will denote $t_{n,k} = \sum_{i=1}^{k}m(E_{n,i})$ for the rest of the proof.

Now, define a function
\[	
f_n = \sum_{k=1}^{\infty} u_{n,k} \chi_{E_{n,k}}
\]
Notice that $f_n^* = \sum_{k=1}^{\infty} u_{n,k} \chi_{[t_{n,k}, t_{n,k+1})}$. Let
\[
f:= \sum_{n=1}^{\infty}\frac{f_n}{2^n\|f_n\|},
\]
whose support is the set $E$. For every $\lambda > 0$, by the orthogonal subadditivity and convexity of $\rho_{\varphi,w}$ and (\ref{eq:wpieces1}) there exists $k_n > \frac{\lambda}{\|f_n\|}$ such that
\begin{eqnarray*}
	\rho_{\varphi,w}(\lambda f) &\leq& \sum_{n=1}^{\infty}\frac{1}{2^n}\cdot \rho_{\varphi,w}\left(\frac{\lambda f_n}{\|f_n\|}\right) = \sum_{n=1}^{\infty}\frac{1}{2^n}\cdot\left(\sum_{k=1}^{\infty} \varphi\left(\frac{\lambda u_{n,k}}{\|f_n\|}\right)\int_{t_{n,k}}^{t_{n,k+1}}w\right)\\
	&=&\sum_{n=1}^{\infty}\frac{1}{2^n}\cdot\left(\sum_{k=1}^{k_n}\varphi\left(\frac{\lambda u_{n,k}}{\|f_n\|}\right) \int_{t_{n,k}}^{t_{n,k+1}}w + \sum_{k= k_n + 1}^{\infty}\varphi\left(\frac{\lambda u_{n,k}}{\|f_n\|}\right) \int_{t_{n,k}}^{t_{n,k+1}}w\right)\\
	&\leq& \sum_{n=1}^{\infty}\frac{1}{2^n}\cdot\left( \varphi(k_n u_{n,k_n}) \sum_{k=1}^{k_n}\int_{t_{n,k}}^{t_{n,k+1}}w +  \sum_{k= k_n + 1}^{\infty}\varphi(k u_{n,k}) \int_{t_{n,k}}^{t_{n,k+1}}w\right)\\
	&\leq&
	\sum_{n=1}^{\infty}\frac{1}{2^n}\cdot\left( \varphi(k_n u_{n,k_n}) \sum_{k=1}^{k_n}\int_{t_{n,k}}^{t_{n,k+1}}w +  \sum_{k= k_n + 1}^{\infty}\frac{1}{2^k}\right).
\end{eqnarray*}
Moreover, notice from (\ref{eq:wpieces1}) that $\int_{t_{n,k_n}}^{t_{n,k_n+1}}w \geq \int_{t_{n,k}}^{t_{n,k+1}}w$ for every $k < k_n$. Hence, we obtain
\begin{eqnarray*}
	\rho_{\varphi,w}(\lambda f) &\leq& \sum_{n=1}^{\infty}\frac{1}{2^n}\cdot\left( \varphi(k_n u_{n,k_n}) \sum_{k = 1}^{k_n}\int_{t_{n,k}}^{t_{n,k+1}}w +  \sum_{k= k_n + 1}^{\infty}\frac{1}{2^k}\right)\\
	&\leq&\sum_{n=1}^{\infty}\frac{1}{2^n}\cdot\left(\frac{k_n}{2^{k_n}} + \sum_{k= k_n + 1}^{\infty}\frac{1}{2^k}\right) \leq 1 < \infty.
\end{eqnarray*}
This shows that $f \in X = (\Lambda_{\varphi,w})_a$. On the other hand, for every $\varepsilon > 0$ and for each $n \in \mathbb{N}$, if $\frac{\varepsilon k_n}{2^n \|f_n\|} > 1$ for some $k_n \in \mathbb{N}$, we have by the monotonicity of $\rho_{\varphi_n,w}$ that
\begin{align*}
	\rho_{\varphi_n,w}(\varepsilon f)\geq \rho_{\varphi_n,w}\left(\frac{\varepsilon f_n}{2^n\|f_n\|}\right) & \geq \sum_{k = k_n + 1}^{\infty} \varphi_n\left( \frac{u_{n,k}}{k_n} \right) \int_{t_{n,k}}^{t_{n,k+1}}w \\
	& \geq \sum_{k= k_n + 1}^{\infty} \frac{2^k\varphi(k u_{n,k})}{2^k\varphi(k u_{n,k})} = \infty.
\end{align*}



This shows that $f \notin \Lambda_{\varphi_n,w}$ for all $n \in \mathbb{N}$. Hence, $f \in X \setminus B$ is a function whose support is $E$.

The maximal-spaceability can be obtained by using similar ideas used at the end of the proof of Theorem~\ref{th:OLspaceable1}.
We are done.
\end{proof}

Applying ideas analogous to those used in the proof of Corollary~\ref{co:OLspaceable1}, we have the following corollary.

\begin{corollary}\label{co:OLspaceable2}
	Let $\alpha \geq \aleph_0$ be a cardinal number, $I=[0,\infty)$ and assume that \,$w$ \,is a weight function on $I$.
	Let $\varphi$, $\varphi_n$ $(n \geq 1)$ be Orlicz functions such that each \,$\varphi_n$ \,satisfies the $\Delta_\varphi(0)$-condition and $\varphi_{n+1} \prec \varphi_n$ \,for every $n\in \mathbb N$.
		\begin{itemize}
			\item[\textup{(i)}] If $\alpha\leq \mathfrak c$ or $\alpha > \emph{card} ((\Lambda_{\varphi,w})_a)$, then $(\Lambda_{\varphi,w})_a \setminus \bigcup_{n=1}^\infty \Lambda_{\varphi_n,w}$ is not $(\alpha,\beta)$-spaceable in $(\Lambda_{\varphi,w})_a$, regardless of the cardinal number $\beta \geq \alpha$;
			\item[(ii)] If \,$\varphi$ satisfies the $\Delta_2$-condition, then $\Lambda_{\varphi,w} \setminus \bigcup_{n=1}^\infty \Lambda_{\varphi_n,w}$ is not $(\alpha,\beta)$-spaceable in $\Lambda_{\varphi,w}$, regardless of the cardinal number $\beta \geq \alpha$.
		\end{itemize}
\end{corollary}

We can apply the techniques in Theorems~\ref{th:OLspaceable1} and~\ref{th:OLspaceable2} to show the analogous statement when $(\varphi_n)_{n=1}^{\infty}$ is a mixture of Orlicz functions that satisfy the $\Delta_{\varphi}(0)$-condition or the $\Delta_{\varphi}(\infty)$-condition.

\begin{Theorem}\label{co:OLspaceablemix}
Assume that \,$I = [0,\infty )$ and \,$w$ is a weight function on \,$I$. Let \,$\varphi, \varphi_n$ $(n \geq 1)$ \,be Orlicz functions such that each \,$\varphi_n$ \,satisfies the $\Delta_{\varphi}(0)$-condition or the $\Delta_{\varphi}(\infty)$-condition. Then the set \,$(\Lambda_{\varphi,w})_a \setminus \bigcup_{n=1}^\infty \Lambda_{\varphi_n,w}$ \,is $\mathfrak c$-spaceable in \,$(\Lambda_{\varphi,w})_a$.
Moreover, if $\varphi$ satisfies the $\Delta_2$-condition, then \,$\Lambda_{\varphi,w} \setminus \bigcup_{n=1}^\infty \Lambda_{\varphi_n,w}$ \,is maximal-spaceable in \,$\Lambda_{\varphi,w}$.
\end{Theorem}

\begin{proof}
In view of Theorem \ref{th:pcs}, for a Lebesgue measurable set \,$E \in \Sigma$ \,with \,$m(E) = \infty$, we construct a function $f \in (\Lambda_{\varphi,w})_a \setminus \bigcup_{n=1}^{\infty} \Lambda_{\varphi_n,w}$ whose support is $E$. First, let $F_1, F_2 \subset E$ such that $F_1 \cap F_2 = \varnothing$, $m (F_1) = t_0 < \infty$, and $F_1 \cup F_2 = E$. Also, consider
\begin{align*}
N_1 &:= \{n \in \mathbb{N}: \varphi_n \,\,\, \text{satisfies}\,\,\, \Delta_{\varphi}(\infty)\text{-condition}\}\\
N_2 &:= \{n \in \mathbb{N}: \varphi_n \,\,\, \text{satisfies}\,\,\, \Delta_{\varphi}(0)\text{-condition}\}
\end{align*}
Then for every $a \in N_1$ there exists an increasing sequence $(u_{a,k})_{k=1}^{\infty} \uparrow \infty$ of real numbers such that $\varphi_a(u_{a,k}) \geq 2^k \varphi(k^2 u_{a,k})$. On the other hand, for every $b \in N_2$, there exists a decreasing sequence $(v_{b,l})_{l=1}^{\infty} \downarrow 0$ of real numbers such that
$\varphi_b\left(\frac{v_{b,l}}{l}\right) \geq 2^l\varphi(l v_{b,l})$ \,for all \,$l \in \mathbb{N}$.

Then, passing to subsequences if necessary, using the same arguments as Theorem~\ref{co:OLspaceable1} on the subset $F_1$ and Theorem~\ref{co:OLspaceable2} on the subset $F_2$, we can choose sequences $(u_{a,k})_{k=1}^{\infty}\uparrow \infty$, $(t_{a,k})_{k=1}^{\infty}\downarrow 0$, $(v_{b,l})_{l=1}^{\infty}\downarrow 0$, and $(t_{b,l})_{l=1}^{\infty}\uparrow \infty$ such that
\[
\frac{1}{2^k \varphi(k^2 u_{a,k})} = \int_{t_{a,k+1}}^{t_{a,k}} w \,\,\,\, \text{and} \,\,\,\, \frac{1}{2^l \varphi(l v_{b,l})} = \int_{t_{b,l}}^{t_{b,l+1}} w.
\]
Moreover, we can find the partitions $(F_{a,k})_{k=1}^{\infty}$ of $F_1$ and $(F_{b,l})_{l=1}^{\infty}$ of $F_2$ such that $m(F_{a,k}) = t_{a,k} - t_{a,k+1}$ and
$m(F_{b,l}) = t_{b,l+1} - t_{b,l}$.

Now, define
\[
f_{a} := \sum_{k=1}^{\infty} ku_{a,k}\chi_{E_{a,k}}\,\,\, \text{and}\,\,\, f_{b} := \sum_{l = 1}^{\infty} u_{b,l} \chi_{E_{b,l}}.
\]
We see that $f_a^* = \sum_{k=1}^{\infty}ku_{a,k} \chi_{[t_{a,k+1}, t_{a,k})}$ and $f_b^* = \sum_{l=1}^{\infty} u_{b,l} \chi_{[t_{b,l}, t_{b,l+1})}$. Define
\[
f := \sum_{a\in N_1}\frac{f_a}{2^a\|f_a\|} + \sum_{b\in N_2}\frac{f_b}{2^b\|f_b\|}.
\]
Using the orthogonal subadditivity and convexity of $\rho_{\varphi,w}$, for every $\lambda > 0$, we can find natural numbers \,$k_a > \frac{\lambda}{\|f_a\|}$\, and \,$l_b > \frac{\lambda}{\|f_b\|}$\, such that
\begin{align*}
	\rho_{\varphi,w}(\lambda f) & \leq \sum_{n=1}^{\infty} \frac{1}{2^n}\rho_{\varphi,w} \left( \frac{\lambda f_n}{\|f_n\|}\right) = \sum_{a \in N_1}\frac{1}{2^a}\rho_{\varphi,w} \left( \frac{\lambda f_a}{\|f_a\|} \right) + \sum_{b \in N_2}\frac{1}{2^b}\rho_{\varphi,w} \left( \frac{\lambda f_b}{\|f_b\|} \right)\\
	& = \sum_{a \in N_1}\frac{1}{2^a}\left(\sum_{k=1}^{k_a}\varphi \left( \frac{\lambda k u_{a,k}}{\|f_a\|} \right) \int_{t_{a,k+1}}^{t_{a,k}} w + \sum_{k= k_a + 1}^{\infty}\varphi \left( \frac{\lambda k u_{a,k}}{\|f_a\|} \right) \int_{t_{a,k+1}}^{t_{a,k}} w\right) \\
	&\quad + \sum_{b \in N_2}\frac{1}{2^b}\left(\sum_{l=1}^{l_b}\varphi \left( \frac{\lambda v_{b,l}}{\|f_b\|} \right) \int_{t_{b,l}}^{t_{b,l+1}} w + \sum_{l= l_b + 1}^{\infty}\varphi \left( \frac{\lambda v_{b,l}}{\|f_b\|} \right) \int_{t_{b,l}}^{t_{b,l+1}} w\right)\\
	&\leq \sum_{a\in N_1} \frac{1}{2^a}\cdot\left( \varphi(k_a^2 u_{a,k_a}) \sum_{k=1}^{k_a}\int_{t_{a,k+1}}^{t_{a,k}}w +  \sum_{k= k_a + 1}^{\infty}\varphi(k^2 u_{a,k}) \int_{t_{a,k+1}}^{t_{a,k}}w\right)\\
	& \quad +\sum_{b\in N_2} \frac{1}{2^b}\cdot\left( \varphi(l_b u_{b,l_b}) \sum_{l=1}^{l_b}\int_{t_{b,l}}^{t_{b,l+1}}w +  \sum_{l= l_b + 1}^{\infty}\varphi(l u_{b,l}) \int_{t_{b,l}}^{t_{b,l+1}}w\right)\\
	&\leq
	\sum_{a\in N_1} \frac{1}{2^a}\cdot\left( \varphi(k_a^2 u_{a,k_a}) \sum_{k=1}^{k_a}\int_{t_{a,k+1}}^{t_{a,k}}w +  \sum_{k= k_a + 1}^{\infty}\frac{1}{2^k}\right)\\
	&\quad + \sum_{b\in N_2} \frac{1}{2^b}\cdot\left( \varphi(l_b u_{b,l_b}) \sum_{l=1}^{l_b}\int_{t_{b,l}}^{t_{b,l+1}}w +  \sum_{l= l_b + 1}^{\infty}\frac{1}{2^l}\right).
\end{align*}
Notice that $\int_{t_{a,k_a+1}}^{t_{a,k_a}}w \geq \int_{t_{a,k+1}}^{t_{a,k}}w$ and $\int_{t_{b,l_b}}^{t_{b,l_b+1}}w \geq \int_{t_{b,l}}^{t_{b,l+1}}w$ for all $k < k_a$ and $l < l_b$. Hence we see that
\begin{align*}
\rho_{\varphi,w}(\lambda f) &\leq \sum_{a\in N_1} \frac{1}{2^a}\cdot\left(\frac{k_a}{2^{k_a}} + \sum_{k= k_a + 1}^{\infty}\frac{1}{2^k}\right) + \sum_{b\in N_2} \frac{1}{2^b}\cdot\left(\frac{l_b}{2^{l_b}} + \sum_{l= l_b + 1}^{\infty}\frac{1}{2^l}\right) \leq 2 < \infty
\end{align*}
This shows that $f \in X = (\Lambda_{\varphi,w})_a$.

On the other hand, for any \,$a\in N_1$,\, $\varepsilon > 0$,\, and \,$n \in \mathbb{N}$, if \,$\frac{\varepsilon k_a}{2^a \|f_a\|} > 1$ for some $k_a \in \mathbb{N}$, by the monotonicity of \,$\rho_{\varphi_a,w}$\, we have
\begin{eqnarray*}
\rho_{\varphi_a,w}(\var f) \geq \rho_{\varphi_a,w} \left( \frac{\var f_a}{2^a \|f_a\|} \right)
&\geq& \sum_{k = k_a + 1}^{\infty} \varphi_a(u_{a,k}) \int_{t_{a,k+1}}^{t_{a,k}} w\\
&\geq& \sum_{k= k_a + 1}^{\infty} \frac{2^k\varphi (k^2 u_{a,k})}{2^k\varphi (k^2 u_{a,k})} = \infty.
\end{eqnarray*}

Hence $f \notin \Lambda_{\varphi_a,w}$ for every $a \in N_1$.

On the other hand, for any \,$b\in N_2$,\, $\varepsilon > 0$,\, and \,$n \in \mathbb{N}$, if \,$\frac{\varepsilon l_b}{2^b \|f_b\|} > 1$ for some $l_b \in \mathbb{N}$, we also see by the monotonicity of $\rho_{\varphi_b,w}$ that
\begin{eqnarray*}
\rho_{\varphi_b,w}(\var f) \geq \rho_{\varphi_b,w} \left( \frac{\var f_b}{2^b \|f_b\|} \right)
&\geq& \sum_{l = l_b +1}^{\infty} \varphi_b\left(\frac{v_l}{l}\right) \int_{t_{b,l}}^{t_{b,l+1}}  w\\
&\geq& \sum_{l= l_b +1}^{\infty} \frac{2^l\varphi(l v_l)}{2^l\varphi(l v_l)} = \infty.
\end{eqnarray*}

This shows that $f \notin \Lambda_{\varphi_b, w}$ for every $b \in N_2$.
Consequently, we have $f \in X \setminus B = (\Lambda_{\varphi,w})_a \setminus \bigcup_{n=1}^{\infty}\Lambda_{\varphi_n,w}$ whose support is $E$.

The maximal-spaceability property of the second part can be obtained by using similar ideas used at the end of the proof of Theorem~\ref{th:OLspaceable1} (and~\ref{th:OLspaceable2}).
The proof is finished.
\end{proof}

Similarly to Corollaries~\ref{co:OLspaceable1} and~\ref{co:OLspaceable2}, we have the following result.

\begin{corollary}\label{co:OLspaceable3}
	Let $\alpha \geq \aleph_0$ be a cardinal number, $I = [0,\infty )$ and assume that \,$w$ is a weight function on \,$I$.
	Let \,$\varphi, \varphi_n$ $(n \geq 1)$ \,be Orlicz functions such that each \,$\varphi_n$ \,satisfies the $\Delta_{\varphi}(0)$-condition or the $\Delta_{\varphi}(\infty)$-condition, and \,$\varphi_{n+1} \prec \varphi_n$ \,for every $n\in \mathbb N$.
		\begin{itemize}
			\item[\textup{(i)}] If $\alpha\leq \mathfrak c$ or $\alpha > \emph{card} ((\Lambda_{\varphi,w})_a)$, then $(\Lambda_{\varphi,w})_a \setminus \bigcup_{n=1}^\infty \Lambda_{\varphi_n,w}$ is not $(\alpha,\beta)$-spaceable in $(\Lambda_{\varphi,w})_a$, regardless of the cardinal number $\beta \geq \alpha$;
			\item[(ii)] If \,$\varphi$ satisfies the $\Delta_2$-condition, then $\Lambda_{\varphi,w} \setminus \bigcup_{n=1}^\infty \Lambda_{\varphi_n,w}$ is not $(\alpha,\beta)$-spaceable in $\Lambda_{\varphi,w}$, regardless of the cardinal number $\beta \geq \alpha$.
		\end{itemize}
\end{corollary}

As an application of Theorems~\ref{th:OLspaceable1} and~\ref{th:OLspaceable2}, we introduce the following subsets of the Lorentz space $\Lambda_{p,w} = \Lambda_{p,w}(I)$ that generalize the set of (left, right) strictly $p$-integrable functions on $I$:
\begin{align*}
	\Lambda_{p,w}^{\text{l-strict}} &= \Lambda_{p,w} \setminus \bigcup_{q \in [1, p)}\Lambda_{q,w} = \Lambda_{p,w} \setminus \bigcup_{n=k}^{\infty}\Lambda_{p-\frac{1}{n},w} \,\,\, \text{for some}\,\,\, k \in \mathbb{N}\\
	\Lambda_{p,w}^{\text{r-strict}} &= \Lambda_{p,w} \setminus \bigcup_{q \in [p, \infty)}\Lambda_{q,w} = \Lambda_{p,w} \setminus \bigcup_{n=1}^{\infty}\Lambda_{p+\frac{1}{n},w}\\
	\Lambda_{p,w}^{\text{strict}} &= \Lambda_{p,w}^{\text{l-strict}} \cap \Lambda_{p,w}^{\text{r-strict}}.
\end{align*}

When $\varphi(u) = u^p$, $1 \leq p < \infty$, the Orlicz-Lorentz space becomes the Lorentz space $\Lambda_{p,w}$. Since the Orlicz function \,$\varphi(u) = u^p$ \,satisfies the appropriate $\Delta_2$-condition, the Lorentz space $\Lambda_{p,w}$ is order-continuous \cite{K}, and so $(\Lambda_{\varphi,w})_a = (\Lambda_{p,w})_a = \Lambda_{p,w}$. It is well-known that $\psi(u) = u^q$ satisfies the $\Delta_{\varphi}(0)$- (resp. $\Delta_{\varphi}(\infty)$-) condition if, and only if, \,$q < p$\, (resp. \,$p < q$). Hence, we can see that $\Lambda_{p,w}^{\text{l-strict}} = (\Lambda_{\varphi,w})_a \setminus \bigcup_{n=k}^{\infty} \Lambda_{\varphi_n,w}$ for Orlicz functions $\varphi_n(u) =  u^{p - \frac{1}{n}}$ and a natural number $k > \frac{1}{p-1}$. Similarly, $\Lambda_{p,w}^{\text{r-strict}} = (\Lambda_{\varphi,w})_a \setminus \bigcup_{n=1}^{\infty} \Lambda_{\varphi_n,w}$ for Orlicz functions $\varphi_n(u) = u^{p + \frac{1}{n}}$. As a consequence, these subsets of $\Lambda_{p,w}$ are also maximal-spaceable (specifically $\mathfrak c$-spaceable) by Theorems~\ref{th:OLspaceable1} and~\ref{th:OLspaceable2}:

\begin{Corollary}\label{corollary Lambda p w strict}
	Let \,$w$ \,be a weight function on $I$. Then the following statements hold:
	\begin{enumerate}[\rm(i)]
		\item For $p > 1$, the set $\Lambda_{p,w}^{\text{l-strict}}$ \,is maximal-spaceable {\rm (}specifically $\mathfrak c$-spaceable{\rm )} in \,$\Lambda_{p,w}$.
		\item For $p < \infty$, the set $\Lambda_{p,w}^{\text{r-strict}}$ \,is maximal-spaceable {\rm (}specifically $\mathfrak c$-spaceable{\rm )} in \,$\Lambda_{p,w}$.
		\item For $p \in (1, \infty)$, the set $\Lambda_{p,w}^{\text{strict}}$ \,is maximal-spaceable {\rm (}specifically $\mathfrak c$-spaceable{\rm )} in \,$\Lambda_{p,w}$.
	\end{enumerate}
\end{Corollary}

Given $w$ and $(w_n)_{n=1}^\infty$ weight functions, we can derive the results analogous to those in Theorems~\ref{th:OLspaceable1}, \ref{th:OLspaceable2} and~\ref{co:OLspaceablemix} under appropriate conditions on the weight functions $w$ and $(w_n)_{n=1}^\infty$.
To show this, let us prove first the following lemma which is interesting on its own.
	
\begin{Lemma}\label{lem:OLspaceabilityspan}
	Let $I=[0,\gamma)$ with $\gamma \leq \infty$, $\varphi$ be an Orlicz function that satisfies the $\Delta_2$-condition for $\gamma = \infty$ ($\Delta_2^{\infty}$-condition for $\gamma < \infty$) and $w$ and $(w_n)_{n=1}^\infty$ weight functions on $I$. We denote \,$W(t) = \int_{0}^{t} w(s) \,ds$ and \,$W_n(t) = \int_{0}^{t} w_n(s) \,ds$ for each $n\geq 1$.
	If the following two statements hold:
		\begin{enumerate}[\rm(i)]
			\item there are positive constants \,$(K_n)_{n=1}^\infty$ \,such that \,$W(t) \leq K_n \cdot W_n(t)$ \,for all \,$t \in I$ and \,$n\in \mathbb N$, and
			\item $\Lambda_{\varphi,w_n} \subsetneq \Lambda_{\varphi,w_{n+1}}$ \,for every \,$n\in \mathbb N$,
		\end{enumerate}
	then the set \,$\bigcup_{n=1}^\infty \Lambda_{\varphi,w_n}$ is a non-closed dense vector subspace of \,$\Lambda_{\varphi,w}$.
\end{Lemma}

\begin{proof}
	Item (i) implies that the identity operator \,$\text{Id}_n : \Lambda_{\varphi,w_n} \to \Lambda_{\varphi,w}$ \,is continuous for every $n\in \mathbb N$\, by Theorem~\ref{Thm Orl-Lor inclusion-2}.
	Hence, \,$\text{Id}_n (\Lambda_{\varphi,w_n}) = \Lambda_{\varphi,w_n} \subset \Lambda_{\varphi,w}$\, for every $n\in \mathbb N$.
	Now it is easy to see that the latter combined with item (ii) yield that \,$\bigcup_{n=1}^\infty \Lambda_{\varphi,w_n}$ is a vector subspace of \,$\Lambda_{\varphi,w}$.
	
	Now let us show that \,$\bigcup_{n=1}^\infty \Lambda_{\varphi,w_n}$\, is not closed in \,$\Lambda_{\varphi,w}$.
	Assume to the contrary that $\bigcup_{n=1}^\infty \Lambda_{\varphi,w_n}$ is closed in \,$\Lambda_{\varphi,w}$.
	By \cite[Proposition~3.1]{KT} applied to \,$Y := \bigcup_{n=1}^\infty \text{Id}_n (\Lambda_{\varphi,w_n}) = \bigcup_{n=1}^\infty \Lambda_{\varphi,w_n}$ \,and item (ii), there exists \,$m \in \mathbb N$ such that
		$$
		Y = \bigcup_{n=1}^\infty \Lambda_{\varphi,w_n} = \bigcup_{n=1}^m \Lambda_{\varphi,w_n} = \Lambda_{\varphi,w_m}.
		$$
	Note that by item (ii), there is an \,$f \in \Lambda_{\varphi,w_{m+1}} \setminus \Lambda_{\varphi,w_m} = \Lambda_{\varphi,w_{m+1}} \setminus Y $.
	But by definition of \,$Y$ we have that \,$f\in Y$, which is a contradiction.
	
	Finally, it remains to prove that \,$\bigcup_{n=1}^\infty \Lambda_{\varphi,w_n}$ is a dense in \,$\Lambda_{\varphi,w}$.
	Since $\varphi$ satisfies the appropriate \,$\Delta_2$-condition, the set of simple functions (call it $S$) is dense in \,$\Lambda_{\varphi,w}$\, and \,$\Lambda_{\varphi,w_n}$\, for every \,$n\in \mathbb N$.
	Moreover, observe that
		$$
		S \subset \Lambda_{\varphi,w_1} \subset \bigcup_{n=1}^\infty \Lambda_{\varphi,w_n} \subset \Lambda_{\varphi,w}.
		$$
	which yields that \,$\bigcup_{n=1}^\infty \Lambda_{\varphi,w_n}$ is dense in \,$\Lambda_{\varphi,w}$, as needed.
\end{proof}

Now we are ready to prove the main result.

\begin{Theorem}\label{thm:OLspaceabilityspan1}
	Let $I=[0,\gamma)$ with $\gamma \leq \infty$, $\varphi$ be an Orlicz function that satisfies the $\Delta_2$-condition for $\gamma = \infty$ ($\Delta_2^{\infty}$-condition for $\gamma < \infty$) and $w$ and $(w_n)_{n=1}^\infty$ be weight functions on $I$. We denote \,$W(t) = \int_{0}^{t} w(s) \,ds$ and \,$W_n(t) = \int_{0}^{t} w_n(s) \,ds$ for each $n\geq 1$.
	If the following two statements hold:
		\begin{enumerate}[\rm(i)]
			\item there are positive constants \,$(K_n)_{n=1}^\infty$ \,such that \,$W(t) \leq K_n \cdot W_n(t)$ \,for all \,$t \in I$ and \,$n\in \mathbb N$, and
			\item $\Lambda_{\varphi,w_n} \subsetneq \Lambda_{\varphi,w_{n+1}}$ \,for every \,$n\in \mathbb N$,
		\end{enumerate}
	then \,$\Lambda_{\varphi,w} \setminus \bigcup_{n=1}^\infty \Lambda_{\varphi,w_n}$ is maximal-spaceable (specifically $\mathfrak c$-spaceable) in \,$\Lambda_{\varphi,w}$.
\end{Theorem}

\begin{proof}
	By Lemma~\ref{lem:OLspaceabilityspan} and \cite[Theorem~3.3]{KT}, we have that \,$\Lambda_{\varphi,w} \setminus \bigcup_{n=1}^\infty \Lambda_{\varphi,w_n}$ is spaceable in \,$\Lambda_{\varphi,w}$.
	
	Now, since \,$\varphi$ \,satisfies the \,$\Delta_2$-condition for \,$\gamma = \infty$ ($\Delta_2^{\infty}$-condition for \,$\gamma < \infty$), the space \,$\Lambda_{\varphi,w} = (\Lambda_{\varphi,w})_a$ is an infinite-dimensional separable Banach space.
	Then, applying similar arguments used in the previous proofs, the maximal spaceability (and specifically the $\mathfrak c$-spaceability) property follows.
	The proof is done.
\end{proof}

From Theorem~\ref{thm:OLspaceabilityspan1} we can derive the following corollary which gives us more information on the algebraic structures inside \,$\Lambda_{\varphi,w} \setminus \bigcup_{n=1}^\infty \Lambda_{\varphi,w_n}$\, when compared to Corollaries~\ref{co:OLspaceable1}, \ref{co:OLspaceable2} and \ref{co:OLspaceable3}.

\begin{corollary}\label{cor:OLspaceabilityspan1}
	Let \,$I=[0,\gamma)$ with $\gamma \leq \infty$, $\varphi$ be an Orlicz function that satisfies the $\Delta_2$-condition for $\gamma = \infty$ ($\Delta_2^{\infty}$-condition for $\gamma < \infty$) and $w$ and $(w_n)_{n=1}^\infty$ be weight functions on $I$. 
	We denote \,$W(t) = \int_{0}^{t} w(s) \,ds$ and \,$W_n(t) = \int_{0}^{t} w_n(s) \,ds$ for each $n\geq 1$.
	Assume that the following two statements hold:
	\begin{enumerate}[\rm(i)]
			\item there are positive constants \,$(K_n)_{n=1}^\infty$ \,such that \,$W(t) \leq K_n \cdot W_n(t)$ \,for all \,$t \in I$ and \,$n\in \mathbb N$, and
			\item $\Lambda_{\varphi,w_n} \subsetneq \Lambda_{\varphi,w_{n+1}}$ \,for every \,$n\in \mathbb N$.
	\end{enumerate}
	Then the following assertions are satisfied.
	\begin{enumerate}[\rm(I)]
		\item $\Lambda_{\varphi,w} \setminus \bigcup_{n=1}^\infty \Lambda_{\varphi,w_n}$ is $(\alpha,\beta)$-dense-lineable in \,$\Lambda_{\varphi,w}$ for every cardinal number $\alpha$ and $\beta$ such that \,$\alpha < \mathfrak c$ and $\max\{ \aleph_0,\alpha \} \leq \beta \leq \mathfrak c$.
		\item $\Lambda_{\varphi,w} \setminus \bigcup_{n=1}^\infty \Lambda_{\varphi,w_n}$ is pointwise maximal dense-lineable {\rm (}specifically pointwise $\mathfrak c$-dense-lineable{\rm )} in \,$\Lambda_{\varphi,w}$.
		\item If $\alpha \geq \aleph_0$ is a cardinal number, then \,$\Lambda_{\varphi,w} \setminus \bigcup_{n=1}^\infty \Lambda_{\varphi,w_n}$ is not \,$(\alpha,\beta)$-spaceable in \,$\Lambda_{\varphi,w}$, regardless of the cadinal number $\beta \geq \alpha$.
		\item If $\alpha \leq \mathfrak c$ is a cardinal number, then \,$\Lambda_{\varphi,w} \setminus \bigcup_{n=1}^\infty \Lambda_{\varphi,w_n}$ is \,$(\alpha,\mathfrak c)$-lineable in \,$\Lambda_{\varphi,w}$.
	\end{enumerate}
\end{corollary}

\begin{proof}
	From the fact that \,$\Lambda_{\varphi,w}$ is separable, it has density character $\aleph_0$ and so  $\textup{codim} \left( \bigcup_{n=1}^\infty \Lambda_{\varphi,w_n} \right)=\mathfrak c$. Hence $\Lambda_{\varphi,w} \setminus \bigcup_{n=1}^\infty \Lambda_{\varphi,w_n}$ being $\aleph_0$-lineable by Theorem~\ref{thm:OLspaceabilityspan1} and \cite[Corollary~4.3]{FPRR}
	(The codimension is obtained by using the fact that \,$\Lambda_{\varphi,w}$ has dimension \,$\mathfrak c$, Theorem~\ref{thm:OLspaceabilityspan1} and proof of \cite[Corollary~2]{AB}). This shows (I).
	
	(II) is an immediate consequence of \cite[Corollary~4.4]{FPRR}.
	
	For (III), assume first that $\alpha \leq \mathfrak c$.
	Recall that on metric spaces the weight of a topological space coincides with its density character, so the weight of \,$\Lambda_{\varphi,w}$\, is \,$\aleph_0$.
	As \,$\bigcup_{n=1}^\infty \Lambda_{\varphi,w_n}$\, is a proper subspace of \,$\Lambda_{\varphi,w}$\, and $\textup{codim}\left(\bigcup_{n=1}^\infty \Lambda_{\varphi,w_n} \right) = \mathfrak c$, the result is derived from \cite[Theorem~9]{AB2}.
	The case when $\alpha > \mathfrak c$ follows from \,$\Lambda_{\varphi,w}$ having cardinality \,$\mathfrak c$.
	
	Finally, (IV) is obtained from the codimension of \,$\bigcup_{n=1}^\infty \Lambda_{\varphi,w_n}$ and \cite[Theorem~1]{AB}.
\end{proof}

Using arguments similar to those used in the proofs of Lemma~\ref{lem:OLspaceabilityspan}, Theorem~\ref{thm:OLspaceabilityspan1} and Corollary~\ref{cor:OLspaceabilityspan1} and in conjunction with Theorem~\ref{Thm Orl-Lor inclusion}, we have the following three results.
Theorem~\ref{thm:OLspaceabilityspan2} is another version of the second part of Theorems~\ref{th:OLspaceable1}, \ref{th:OLspaceable2} and~\ref{co:OLspaceablemix}, and Corollary~\ref{cor:OLspaceabilityspan2} improves (ii) of Corollaries~\ref{co:OLspaceable1}, \ref{co:OLspaceable2} and~\ref{co:OLspaceable3}, under appropriate conditions on the Orlicz functions \,$\varphi$ and $(\varphi_n)_{n=1}^\infty$.

\begin{Lemma}\label{lem:OLspaceabilityspan2}
	Let $I=[0,\gamma)$ with $\gamma \leq \infty$, $w$ be a weight function on \,$I$ and
	 \,$\varphi$ and $(\varphi_n)_{n=1}^\infty$ be Orlicz functions that satisfy the $\Delta_2$-condition for $\gamma = \infty$ {\rm (}$\Delta_2^{\infty}$-condition for $\gamma < \infty${\rm )}.
	If the following two statements hold:
	\begin{enumerate}[\rm(i)]
		\item $\varphi \prec \varphi_n$ \,for every $n\in \mathbb N$ {\rm (}resp.~$\varphi \prec_\infty \varphi_n$ \,for every $n\in \mathbb N${\rm )} \,for \,$\gamma = \infty$ \,{\rm (}resp.~$\gamma < \infty${\rm )}, and
		\item $\Lambda_{\varphi_n,w} \subsetneq \Lambda_{\varphi_{n+1},w}$ \,for every \,$n\in \mathbb N$,
	\end{enumerate}
	then the set \,$\bigcup_{n=1}^\infty \Lambda_{\varphi_n,w}$ is a non-closed dense vector subspace of \,$\Lambda_{\varphi,w}$.
\end{Lemma}

\begin{Theorem}\label{thm:OLspaceabilityspan2}
	Let $I=[0,\gamma)$ with $\gamma \leq \infty$, $w$ be a weight function on \,$I$ and
	\,$\varphi$ and $(\varphi_n)_{n=1}^\infty$ be Orlicz functions that satisfy the $\Delta_2$-condition for $\gamma = \infty$ {\rm (}$\Delta_2^{\infty}$-condition for $\gamma < \infty${\rm )}.
	If the following two statements hold:
	\begin{enumerate}[\rm(i)]
		\item $\varphi \prec \varphi_n$ \,for every $n\in \mathbb N$ {\rm (}resp.~$\varphi \prec_\infty \varphi_n$ \,for every $n\in \mathbb N${\rm )} \,for \,$\gamma = \infty$ \,{\rm (}resp.~$\gamma < \infty${\rm )}, and
		\item $\Lambda_{\varphi_n,w} \subsetneq \Lambda_{\varphi_{n+1},w}$ \,for every \,$n\in \mathbb N$,
	\end{enumerate}
	then \,$\Lambda_{\varphi,w}\setminus \bigcup_{n=1}^\infty \Lambda_{\varphi_n,w}$ is maximal-spaceable {\rm (}specifically $\mathfrak c$-spaceable{\rm )} in \,$\Lambda_{\varphi,w}$.
\end{Theorem}

\begin{corollary}\label{cor:OLspaceabilityspan2}
	Let $I=[0,\gamma)$ with $\gamma \leq \infty$, $w$ be a weight function on \,$I$ and
	\,$\varphi$ and $(\varphi_n)_{n=1}^\infty$ be Orlicz functions that satisfy the $\Delta_2$-condition for $\gamma = \infty$ {\rm (}$\Delta_2^{\infty}$-condition for $\gamma < \infty${\rm )}.
	Assume that the following two statements hold:
	\begin{enumerate}[\rm(i)]
		\item $\varphi \prec \varphi_n$ \,for every $n\in \mathbb N$ {\rm (}resp.~$\varphi \prec_\infty \varphi_n$ \,for every $n\in \mathbb N${\rm )} \,for \,$\gamma = \infty$ \,{\rm (}resp.~$\gamma < \infty${\rm )}, and
		\item $\Lambda_{\varphi_n,w} \subsetneq \Lambda_{\varphi_{n+1},w}$ \,for every \,$n\in \mathbb N$.
	\end{enumerate}
	Then the following assertions are satisfied.
	\begin{enumerate}[\rm(I)]
		\item $\Lambda_{\varphi,w} \setminus \bigcup_{n=1}^\infty \Lambda_{\varphi_n,w}$ is $(\alpha,\beta)$-dense-lineable in \,$\Lambda_{\varphi,w}$ for every cardinal number $\alpha$ and $\beta$ such that \,$\alpha < \mathfrak c$ and $\max\{ \aleph_0,\alpha \} \leq \beta \leq \mathfrak c$.
		\item $\Lambda_{\varphi,w} \setminus \bigcup_{n=1}^\infty \Lambda_{\varphi_n,w}$ is pointwise maximal dense-lineable {\rm (}specifically pointwise $\mathfrak c$-dense-lineable{\rm )} in \,$\Lambda_{\varphi,w}$.
		\item If $\alpha \geq \aleph_0$ is a cardinal number, then $\Lambda_{\varphi,w} \setminus \bigcup_{n=1}^\infty \Lambda_{\varphi_n,w}$ is not \,$(\alpha,\beta)$-spaceable in \,$\Lambda_{\varphi,w}$, regardless of the cadinal number $\beta \geq \alpha$.
		\item If $\alpha \leq \mathfrak c$ is a cardinal number, then \,$\Lambda_{\varphi,w} \setminus \bigcup_{n=1}^\infty \Lambda_{\varphi_n,w}$ is \,$(\alpha,\mathfrak c)$-lineable in \,$\Lambda_{\varphi,w}$.
	\end{enumerate}
\end{corollary}

For nonseparable Orlicz-Lorentz spaces, the complement of the order-continuous subspace is not only spaceable, but also contains an isometric copy of $\ell_{\infty}$.

\begin{Theorem} \label{Thm isomorphic copy}
	Let \,$I=[0,\gamma)$ with \,$\gamma \leq \infty$, $\varphi$ be an Orlicz function, and \,$w$ be a weight function. If \,$\varphi$ \,does not satisfy the $\Delta_2$-condition when $\gamma = \infty$
	{\rm (}resp.~$\Delta_2^{\infty}$-condition when $\gamma < \infty${\rm )}, then the set \,$\Lambda_{\varphi,w} \setminus (\Lambda_{\varphi,w})_a$ is $\mathfrak c$-spaceable in \,$\Lambda_{\varphi,w}$.
	Furthermore, the set \,$(\Lambda_{\varphi,w} \setminus (\Lambda_{\varphi,w})_a)\cup \{0\}$ \,contains an isometric copy of \,$\ell_{\infty}$.
\end{Theorem}

\begin{proof}
	Once again, for a given Lebesgue measurable set \,$E \in \Sigma$ \,with \,$m(E) > 0$, we construct a function \,$f \in \Lambda_{\varphi,w} \setminus (\Lambda_{\varphi,w})_a$ \,whose support is contained in $E$. Since \,$\varphi$ \,does not satisfy the $\Delta_2$-condition, $\varphi$ does not satisfy either $\Delta_2^{\infty}$-condition or $\Delta_2^0$-condition. We only consider the case when \,$\varphi$ \,does not satisfy the $\Delta_2^{\infty}$-condition because the other one can be proved by the analogous argument. By Lemma \ref{le:C} there exists an increasing sequence $(u_k)_{k=1}^{\infty}\uparrow \infty$ of real numbers such that
	\begin{equation}
		\varphi\left(\left(1 + \frac{1}{k}\right) u_k \right) > 2^k \varphi(u_k).
	\end{equation}
	
	Let $E \in \Sigma$ and $t_0$ be a positive real number such that $\int_0^{t_0} w < \int_0^{m(E)} w \leq \gamma$. Passing to a subsequence if necessary, we can choose $(u_k)_{k=1}^{\infty}$ such that
	\begin{equation}
		\sum_{k=1}^{\infty} \frac{1}{2^k \varphi(u_k)} = \int_0^{t_0} w.
	\end{equation}
	
	Then there exists a decreasing sequence $(t_k)_{k=1}^{\infty} \downarrow 0$ such that
	\begin{equation}
		\frac{1}{2^k \varphi(u_k)} = \int_{t_{k}}^{t_{k-1}}w.
	\end{equation}	
	
	Now, define
	\[
	f = \sum_{k=1}^{\infty} u_k \chi_{[t_{k},t_{k-1}]}.
	\]
	Again, it is easy to see that $f = f^*$. Furthermore,
	\[
	\rho_{\varphi,w}(f) = \sum_{k=1}^{\infty} \varphi(u_k)\int_{t_{k}}^{t_{k-1}} w = \sum_{k = 1}^{\infty} \varphi(u_k) \cdot \frac{1}{2^k \varphi(u_k)} = 1 < \infty.
	\]
	Hence, $f \in \Lambda_{\varphi,w}$.
	
	On the other hand, for every $\varepsilon > 1$, there exists $l \in \mathbb{N}$ such that $1 + \frac{1}{k} < \varepsilon$ for all $k \geq l$. Then we have
	\begin{align*}
		\rho_{\varphi,w}(\varepsilon f) & \geq  \sum_{k = l + 1}^{\infty} \varphi(\var u_k) \int_{t_{k-1}}^{t_k}w \\
		&\geq  \sum_{k = l + 1}^{\infty} \varphi\left(\left(1 + \frac{1}{k}\right)u_k\right) \int_{t_{k-1}}^{t_k}w\\ & > \sum_{k = l + 1}^{\infty} \frac{2^k\varphi(u_k)}{2^k\varphi(u_k)} = \infty,
	\end{align*}
	and so $f \notin (\Lambda_{\varphi,w})_a$. Hence, $f \in \Lambda_{\varphi,w} \setminus (\Lambda_{\varphi,w})_a$ whose support is $(0, t_0) \subset (0, m(E))$. In view of \cite[Corollary 2.7.8]{BS}, we can find a Lebesgue measurable function $\widetilde{f}$ on the set $E_{t_0} \subset E$ with $m(E_{t_0}) = t_0$ such that $\left(\widetilde{f}\right)^* = f^*$. This shows the first claim by Theorem \ref{th:pcs}(v).
	
	Regarding the second part, define the operator \,$T: \ell_{\infty} \rightarrow (\Lambda_{\varphi,w}\setminus (\Lambda_{\varphi,w})_a)\cup \{0\}$ \,by \,$Tx = \sum_{n=1}^{\infty} x_n f_n$
	where $x = (x_n)_{n=1}^{\infty} \in \ell_{\infty}$ and $(f_n)_{n=1}^{\infty} \subset \Lambda_{\varphi,w} \setminus (\Lambda_{\varphi,w})_a$ constructed by the procedure in the first part on a family of pairwise disjoint measurable subsets $(E_n)_{n=1}^{\infty} \subset \Sigma$. As we saw in the first part, we have $\rho_{\varphi,w}(f_n) \leq 1$ for every $n \in \mathbb{N}$ and so $\|f_n\|_{\varphi,w} \leq 1$. Hence, it is easy to see that $\|Tx\|_{\varphi,w} \leq \|x\|_{\infty}$. Let $0 < \var < \|x\|_{\infty}$, where \,$x \ne 0$. Then there exists $n_0  \in \mathbb{N}$ such that $|x_{n_0}| > \|x\|_{\infty} - \frac{\var}{2}$.
	Thus, by the monotonicity of $\rho_{\varphi,w}$, we have
	\[
	\rho_{\varphi,w}\left(\frac{\sum_{n=1}^{\infty}|x_n f_n|}{\|x\|_{\infty} - \var/2}\right) \ge \rho_{\varphi,w}\left(\frac{|x_{n_0} f_{n_0}|}{\|x\|_{\infty} - \var/2}\right) = \infty.
	\]
	Since $\var > 0$ is arbitrary, this implies that $\|x\|_{\infty} \leq \|Tx\|_{\varphi,w}$. Therefore, the operator $T$ is an isometry, which in turn shows that  $(\Lambda_{\varphi,w}\setminus (\Lambda_{\varphi,w})_a)\cup \{0\}$ has an isometric copy of $\ell_{\infty}$.	
\end{proof}

To finish the paper, we turn to the context of topological genericity.
The arguments in Theorems~\ref{th:OLspaceable1} and~\ref{th:OLspaceable2} also show that the set $(\Lambda_{\varphi,w})_a \setminus \bigcup_{n=1}^{\infty} \Lambda_{\varphi_n,w}$ is either empty or residual under certain assumptions on the $\varphi_n$'s.
	
\begin{Theorem}\label{th:OLresidual}
Let $I=[0,\gamma)$ with $\gamma \leq \infty$ and assume that $w$ is a weight function on $I$.
Let \,$\varphi$, $\varphi_n$ {\rm (}$n \geq 1${\rm )} \,be Orlicz functions such that every \,$\varphi_n$ satisfies either the
$\Delta_{\varphi}(0)$-condition or the $\Delta_{\varphi}(\infty)$-condition.
Then \,$(\Lambda_{\varphi,w})_a \setminus \bigcup_{n=1}^{\infty}\Lambda_{\varphi_n,w}$ is either empty or residual in $(\Lambda_{\varphi,w})_a$.
	\end{Theorem}

\begin{proof}
Let $A_{n,k} = \{f \in (\Lambda_{\varphi,w})_a : \|f\|_{\varphi_n,w} \leq k\}$. Then, we derive that
		\begin{eqnarray*}
		(\Lambda_{\varphi,w})_a \setminus \bigcup_{n=1}^{\infty}\Lambda_{\varphi_n,w} &=& (\Lambda_{\varphi,w})_a \setminus \bigcup_{n=1}^{\infty}((\Lambda_{\varphi,w})_a \cap \Lambda_{\varphi_n,w})\\
		&=& (\Lambda_{\varphi,w})_a \setminus \bigcup_{n,k=1}^{\infty}\{f \in (\Lambda_{\varphi,w})_a : \|f\|_{\varphi_n,w} \leq k\}\\
		&=& \bigcap_{n,k = 1}^{\infty}((\Lambda_{\varphi,w})_a \setminus A_{n,k}).
		\end{eqnarray*}
		
Now, assume that $(\Lambda_{\varphi,w})_a \setminus \bigcup_{n=1}^{\infty}\Lambda_{\varphi_n,w} \neq \varnothing$. Then the set $(\Lambda_{\varphi,w})_a \cap \Lambda_{\varphi_n,w}$ is a proper vector subspace of $(\Lambda_{\varphi,w})_a$ because we can always find a function $f \in (\Lambda_{\varphi,w})_a$ such that $f \notin \Lambda_{\varphi_n,w}$ from the proofs of Theorems~\ref{th:OLspaceable1} and~\ref{th:OLspaceable2}, depending on $\varphi$ satisfying either $\Delta_{\varphi}(\infty)$- or $\Delta_{\varphi}(0)$-condition. Hence \,$\text{int}\,((\Lambda_{\varphi,w})_a \cap \Lambda_{\varphi_n,w}) = \varnothing$ \,and, since
\,$A_{n,k} \subset (\Lambda_{\varphi,w})_a \cap \Lambda_{\varphi_n,w}$, we obtain that $\text{int}(A_{n,k}) = \varnothing$.
		
Moreover, the set $A_{n,k}$ is closed. Indeed, choose a sequence of functions $(f_i)_{i=1}^{\infty} \subset (A_{n,k})$ that converges to $f \in (\Lambda_{\varphi,w})_a$. Then there exists a sequence \,$(k_i)_{i=1}^\infty$ \,of positive real numbers such that \,$k_{i} \uparrow \infty$, $k_{i} \geq \frac{1}{\|f_i - f\|_{\varphi,w}}$ \,and
$\int_I \varphi(k_{i}(f_i - f)^*)w \leq 1$ \,for all \,$i \in \N$. Then for $B \subset I$ such that $m(B) < \infty$ and for every $\var > 0$, we have
		\small
		\begin{eqnarray*}
			W(m(\{t \in B : |f_i(t) - f(t)| \geq \var\})) &\leq& W(m(\{t \in (0, m(B)]: (f_i(t) - f(t))^* \geq \var\}))\\
			&=& W\left(m\left(\left\{t \in (0, m(B)]: \frac{\varphi(k_{i}(f_i(t) - f(t))^*)}{\varphi(k_{i}\var)} \geq 1 \right\}\right)\right)\\
			&\leq& \frac{1}{\varphi(k_{i}\var)}\int_0^{m(B)} \varphi(k_{i}(f_i(t) - f(t))^*)w\\
			&\leq& \frac{1}{\varphi(k_{i}\var)}.
		\end{eqnarray*}
		\normalsize
		
Since $k_{i} \uparrow \infty$, we see that $W(m(\{t \in B : |f_i(t) - f(t)| \geq \var\})) \rightarrow 0$ as $i \to \infty$, and so
$m(\{t \in B : |f_i(t) - f(t)| \geq \var\}) \rightarrow 0$ by our assumption on the weight function. This implies that the sequence is actually convergent in measure on the set $B$. Furthermore, by using the $\sigma$-finiteness of the Lebesgue measure, we can show that the sequence is convergent in measure for $I$. In view of \cite[II.2.$11^{\circ}$, p.~67]{KPS}, this implies that $f_i^* \rightarrow f^*$ a.e. By Fatou's Lemma, we see that
		\[
		\int_I \varphi_n\left(\frac{f^*}{k}\right)w = \int_I \varphi_n\left(\frac{\liminf_{i\rightarrow \infty}f_i^*}{k}\right)w \leq \liminf_{i\rightarrow \infty}\int_I \varphi_n\left(\frac{f_i^*}{k}\right)w \leq 1.
		\]
Hence $f \in A_{n,k}$, which shows that $A_{n,k}$ is closed.

This consequently shows that the set $S := (\Lambda_{\varphi,w})_a \setminus \bigcup_{n=1}^{\infty}\Lambda_{\varphi_n,w}$ is a countable intersection of dense open subsets of $(\Lambda_{\varphi,w})_a$. Since \,$(\Lambda_{\varphi,w})_a$ \,is a complete metric space, Baire's theorem tell us that the set \,$S$ \,is residual in $(\Lambda_{\varphi,w})_a$.
\end{proof}

\section{Final remarks and open problems}

\begin{enumerate}
\item[\bf 1.] 
As immediate consequences of Theorems \ref{th:OLspaceable1}, \ref{th:OLspaceable2}, \ref{Thm isomorphic copy}, \ref{thm:OLspaceabilityspan2}, and \ref{th:OLresidual}, and Corollaries \ref{co:OLspaceable1}, \ref{co:OLspaceable2}, \ref{co:OLspaceablemix}, \ref{co:OLspaceable3}, and \ref{cor:OLspaceabilityspan2}, we obtain the analogous statements for the Orlicz spaces $L_{\varphi}(I)$ when $w \equiv 1$. In fact, the more general results for Theorems \ref{th:OLspaceable1}, \ref{th:OLspaceable2}, \ref{Thm isomorphic copy}, and \ref{th:OLresidual} with respect to an abstract measure space with a positive measure are shown in \cite{AM}.



\item[\bf 2.] Regarding Corollary \ref{corollary Lambda p w strict}, note that the subset $\Lambda_{p,w}^{\text{r-strict}}$ 
can be observed in the context of the generalized Lebesgue class, which is known to be a special case of Calder\'on-Lozanovskii spaces. For the interested reader, we refer to \cite{BIM}.
	
\item[\bf 3.] From the materials in Section 3.2.1, we retrieve the results in \cite{Ast} and \cite{RS} if we let $\varphi(u) = u$.

\item[\bf 4.] From the results in Section 3.2, the following problem arises naturally: What is the characterization of inclusion operators between Orlicz-Lorentz spaces to be disjointly strictly singular when the weight function $w$ is fixed? For Orlicz spaces, such characterization has been given in \cite{HR}. However, the proof for Orlicz spaces does not transfer to Orlicz-Lorentz spaces plainly because of technicalities arising from the weight function $w$ and the decreasing rearrangement.


\item[\bf 5.] Using Theorem \ref{Thm Orl-Lor inclusion} and the problem posed at the preceding item, can we make claims about spaceability
or vacuousness of certain subsets of Orlicz-Lorentz spaces similar to the corresponding one given in \cite[Theorem 5.1 and Proposition 5.3]{RS}?

\item[\bf 6.] In the recent paper \cite{BIM}, it is proved (Corollary 3.1) that if \,$X$ \,is a Banach function space, there is a diffuse set (see \cite[Definition 1]{BIM}) \,$E \in \Sigma$ \,with \,$\chi_E \in X$ \,and \,$\mu (E) > 0$, and \,$\varphi , \psi$ \,are Orlicz functions with \,$\varphi \not\prec \psi$, then
    \,$X^\psi \setminus X^\varphi$ \,is spaceable in \,$X^\psi$. Here \,$X^\Phi$ \,denotes the so-called generalized Orlicz class associated to \,$\Phi$ \,and \,$X$
    \,(see, e.g., \cite{Persson}), and Lebesgue, Orlicz and Orlicz-Lorentz spaces are special instances of it, for appropriate pairs \,$(X, \Phi)$. It would be interesting to study to what extent Theorems \ref{th:OLspaceable1}, \ref{th:OLspaceable2}, \ref{thm:OLspaceabilityspan2}, \ref{Thm isomorphic copy} and \ref{th:OLresidual}, and Corollaries \ref{co:OLspaceable1}, \ref{co:OLspaceable2}, \ref{co:OLspaceablemix}, \ref{co:OLspaceable3} and \ref{cor:OLspaceabilityspan2}, hold in this more general context.

\end{enumerate}

\end{document}